\documentclass[egregdoesnotlikesansseriftitles]{scrartcl}

\usepackage[british]{babel}
\usepackage[utf8]{inputenc}
\usepackage[T1]{fontenc}
\usepackage[cal=boondox,bb=boondox]{mathalfa}
\usepackage{euler}
\usepackage{amsmath}
\usepackage{amssymb}
\usepackage{amsthm}
\usepackage{palatino}
\usepackage{enumitem}
\usepackage{etoolbox}
\usepackage{tikz-cd}
\usetikzlibrary{decorations.pathmorphing}
\usepackage{dsfont}
\usepackage{calc}
\usepackage{mathtools}
\usepackage{xstring}
\usepackage{pict2e,picture} 
\usepackage[style=alphabetic,backend=biber,opcittracker=true]{biblatex}
\DeclareFieldFormat{eprint:eudml}{%
  \href{http://eudml.org/doc/#1}{EuDML : #1}%
}
\DeclareFieldFormat{eprint:jstor}{%
  \href{http://www.jstor.org/stable/#1}{JSTOR : #1}%
}
\DeclareFieldFormat{eprint:euclid}{%
  \href{http://projecteuclid.org/euclid.cmp/#1}{Project Euclid : #1}%
}
\usepackage{hyperref}
\usepackage[nameinlink]{cleveref}

\DeclareMathAlphabet{\mathfrak}{U}{jkpmia}{m}{it}
\SetMathAlphabet{\mathfrak}{bold}{U}{jkpmia}{bx}{it}

\newcounter{descriptcount}

\newlist{enumdescript}{description}{1}
\setlist[enumdescript,1]{%
  before={\setcounter{descriptcount}{0}%
          \renewcommand*\thedescriptcount{\arabic{descriptcount}}},
        font={\bfseries\stepcounter{descriptcount}\thedescriptcount.~}
}

\crefname{paragraph}{paragraph}{paragraphs}
\Crefname{Paragraph}{Paragraph}{Paragraphs}
\crefname{subsubsubappendix}{subsubsection}{subsubsections}
\Crefname{subsubsubappendix}{Subsubsection}{Subsubsections}
\crefname{subsubsubsubappendix}{paragraph}{paragraphs}
\Crefname{subsubsubsubappendix}{Paragraph}{Paragraphs}

\theoremstyle{definition}
\newtheorem{definition}[subsubsection]{\definitionautorefname}
\newcommand{\definitionautorefname}{Definition}

\newtheorem{construct}[subsubsection]{\constructautorefname}
\newcommand{\constructautorefname}{Construction}

\theoremstyle{remark}
\newtheorem{remark}[subsubsection]{\remarkautorefname}
\newcommand{\remarkautorefname}{Remark}

\newtheorem{exmp}[subsubsection]{\exmpautorefname}
\newcommand{\exmpautorefname}{Example}

\newcommand{\warningautorefname}{Warning}

\newtheorem{notation}[subsubsection]{\notationautorefname}
\newcommand{\notationautorefname}{Notation}

\theoremstyle{plain} \newtheorem{thm}[subsubsection]{\thmautorefname}
\newcommand{\thmautorefname}{Theorem}

\newtheorem{corlr}[subsubsection]{\corlrautorefname}
\newcommand{\corlrautorefname}{Corollary}

\newcommand{\prorautorefname}{Property}

\newtheorem{pros}[subsubsection]{\prosautorefname}
\newcommand{\prosautorefname}{Proposition}

\newtheorem{lemma}[subsubsection]{\lemmaautorefname}
\newcommand{\lemmaautorefname}{Lemma}

\newtheorem{scholium}[subsubsection]{\scholiumautorefname}
\newcommand{\scholiumautorefname}{Scholium}

\makeatletter
\DeclareRobustCommand\widecheck[1]{{\mathpalette\@widecheck{#1}}}
\def\@widecheck#1#2{%
  \setbox\z@\hbox{\m@th$#1#2$}%
  \setbox\tw@\hbox{\m@th$#1%
    \widehat{%
      \vrule\@width\z@\@height\ht\z@
      \vrule\@height\z@\@width\wd\z@}$}%
  \dp\tw@-\ht\z@
  \@tempdima\ht\z@ \advance\@tempdima2\ht\tw@ \divide\@tempdima\thr@@
  \setbox\tw@\hbox{%
    \raise\@tempdima\hbox{\scalebox{1}[-1]{\lower\@tempdima\box
        \tw@}}}%
  {\ooalign{\box\tw@ \cr \box\z@}}}
\makeatother

\def\slashedrightarrow{\relbar\joinrel\mapstochar\joinrel\rightarrow}

\makeatletter
\DeclareRobustCommand{\bbDelta}{{\mathpalette\bb@Delta\relax}}
\newcommand{\bb@Delta}[2]{%
  \begingroup
  \sbox\z@{$\m@th#1\Delta$}%
  \dimendef\Dht=6 \dimendef\Dwd=8
  \setlength{\Dwd}{\wd\z@}%
  \setlength{\Dht}{\ht\z@}%
  \begin{picture}(\Dwd,\Dht)
  \put(0,0){$\m@th#1\Delta$}
  \put(.42\Dwd,.7\Dht){\line(10,-26){.25\Dht}}
  \end{picture}%
  \endgroup
}
\makeatother
\makeatletter
\DeclareRobustCommand{\bbGamma}{{\mathpalette\bb@Gamma\relax}}
\newcommand{\bb@Gamma}[2]{%
  \begingroup
  \sbox\z@{$\m@th#1\Gamma$}%
  \dimendef\Dht=6 \dimendef\Dwd=8
  \setlength{\Dwd}{\wd\z@}%
  \setlength{\Dht}{\ht\z@}%
  \begin{picture}(\Dwd,\Dht)
  \put(0,0){$\m@th#1\Gamma$}
  \put(.47\Dwd,.025\Dht){\line(0,1){.95\Dht}}
  \end{picture}%
  \endgroup
}
\makeatother
\makeatletter
\DeclareRobustCommand{\bbOmega}{{\mathpalette\bb@Omega\relax}}
\newcommand{\bb@Omega}[2]{%
  \begingroup
  \sbox\z@{$\m@th#1\Omega$}%
  \dimendef\Dht=6 \dimendef\Dwd=8
  \setlength{\Dwd}{\wd\z@}%
  \setlength{\Dht}{\ht\z@}%
  \begin{picture}(\Dwd,\Dht)
  \put(0,0){$\m@th#1\Omega$}
  \put(.2925\Dwd,.167\Dht){\line(0,1){.72\Dht}}
  \put(.6825\Dwd,.17\Dht){\line(0,1){.72\Dht}}
  \end{picture}%
  \endgroup
}
\makeatother
\makeatletter
\DeclareRobustCommand{\bbUpsilon}{{\mathpalette\bb@Upsilon\relax}}
\newcommand{\bb@Upsilon}[2]{%
  \begingroup
  \sbox\z@{$\m@th#1\Omega$}%
  \dimendef\Dht=6 \dimendef\Dwd=8
  \setlength{\Dwd}{\wd\z@}%
  \setlength{\Dht}{\ht\z@}%
  \begin{picture}(\Dwd,\Dht)
  \put(0,0){$\m@th#1\Upsilon$}
  \put(.5125\Dwd,.033\Dht){\line(0,1){.8\Dht}}
  \end{picture}%
  \endgroup
}
\makeatother

\newcommand\mapstofonc{\mathrel{\ooalign{$\rightsquigarrow$\cr%
  \kern-.105ex\raise.325ex\hbox{\scalebox{1}[0.388]{$\mid$}}\cr}}}

\makeatletter
\newcommand{\addchar}[2]{%
  \@tfor\letter:=#1\do{%
    \letter#2
  }%
}
\makeatother

\DeclareMathOperator{\id}{id}%
\newcommand{\cat}[1]{\mathfrak{#1}}
\newcommand{\op}[1]{{#1}^{\mathrm{op}}}
\newcommand{\func}[1]{\mathcal{\addchar{#1}{\!}}\,}
\newcommand{\shname}[1]{\StrSplit{#1}{1}{\debnom}{\resnom}\mathit{\mathscr{\debnom}\resnom}}%
\newcommand{\yo}{\func{y}}

\newcommand{\lmtimes}{\mathop{\times}\limits}

\newcommand{\grothco}[1][{}]{\int^{\mathrm{co}}_{#1}}%
\newcommand{\infgrpds}{\cat{\infty\textnormal{-}Grpd}}
\newcommand{\infcats}{\cat{\mathnormal{(\infty,1)}\textnormal{-}Cat}}
\newcommand{\infoprds}{\cat{\infty\textnormal{-}Oprd}}
\newcommand{\infmoncats}{\cat{Mon\infcats}}
\newcommand{\funcs}[2]{\cat{#2}^{\cat{#1}}}
\newcommand{\prshvs}[1]{\funcs{\op{#1}}{\infgrpds}}
\newcommand{\laxext}[1]{\operatorname{Lex}_{#1}}
\newcommand{\oplaxext}[1]{\operatorname{Opex}_{#1}}
\newcommand{\wtdlim}[2]{\left\{#1,#2\right\}}
\newcommand{\wtdcolim}[2]{#1\star{}#2}%
\newcommand{\seg}[1]{\cat{Seg}_{#1}}
\newcommand{\wkfibs}[1]{\cat{WkSegFib}_{/#1}}
\newcommand{\corol}[1]{\star_{#1}}
\newcommand{\edge}{\eta}
\newcommand{\zeroelem}{\emptyset}

\DeclareMathOperator{\nerve}{\func{N}\!\!_{\bullet}}
\DeclareMathOperator*{\colim}{colim}

\DeclarePairedDelimiter{\uly}{\lvert}{\rvert}
\DeclarePairedDelimiter{\restr}{{}}{\rvert}
\DeclarePairedDelimiter{\bbracks}{[\![}{]\!]}

\title{Monoidal envelopes and Grothendieck construction for dendroidal
  Segal objects} \author{David Kern}

\begin{document}

\maketitle{}

\begin{abstract}
  We propose a construction of the monoidal envelope of
  $\infty$-operads in the model of Segal dendroidal spaces, and use it
  to define cocartesian fibrations of such. We achieve this by viewing
  the dendroidal category as a ``plus construction'' of the category
  of pointed finite sets, and work in the more general language of
  algebraic patterns for Segal conditions. Finally, we rephrase
  Lurie's definition of cartesian structures as exhibiting the
  categorical fibrations coming from envelopes, and deduce a
  straightening/unstraightening equivalence for dendroidal spaces.
\end{abstract}

\tableofcontents{}

\section{Introduction}
\label{sec:introduction}

Any monoidal category $\cat{V}^{\otimes}$ defines an operad whose
colours are the objects of $\cat{V}$ and whose multimorphisms
$C_{1},\cdots,C_{n}\to D$ are given by the morphisms
$C_{1}\otimes\cdots\otimes C_{n}\to D$. The operads which arise in
this way are said to be representable; to avoid confusion --- since we
shall model operads as certain presheaves --- we will call them
\emph{representably monoidal}. Indeed, the tautological multimorphism
$C_{1},\cdots,C_{n}\to C_{1}\otimes\cdots\otimes C_{n}$ (corresponding
to $\id_{C_{1}\otimes\cdots\otimes C_{n}}$) carries the universal
property that every multimorphism with source $(C_{1},\cdots,C_{n})$
must factor through it, by a unique unary morphism of source
$C_{1}\otimes\cdots\otimes C_{n}$. This universal property can be seen
as a cocartesianity condition.

Recall indeed that, if $\func{p}\colon\cat{E}\to\cat{B}$ is an
opfibration (or cocartesian fibration), there is a factorisation
system on $\cat{E}$ whose left class consists of
$\func{p}$-cocartesian morphisms and whose right class consists of
purely $\func{p}$-vertical morphisms. One can also define a notion of
opfibration of multicategories, as done for example
in~\cite{hermida04:_fibrat} for generalised multicategories, and more
generally of cocartesian (multi)morphisms therein, so that a
multicategory is representably monoidal if and only if its morphism to
the terminal operad is an opfibration. In general, an opfibration of
multicagories can be thought of as a morphism whose fibres are
monoidal categories: since the selection of a colour comes from the
operad generated by the nodeless edge $\edge$, which only has unary
morphisms, the vertical arrows are always unary ones, while the
cocartesian arrows are those exhibiting tensor products.

In the homotopical setting, there are various ways of modelling
$\infty$-operads, each coming with its advantages and drawbacks. The
model preferred by~\cite{lurie17:_higher_algeb}, which represents an
$\infty$-operad by its $\infty$-category of operators, is biased to
make cocartesian fibrations and monoidal envelopes immediately
accessible, but requires entangling combinatorics to recover an
operadic intuition on any construction. On the other hand, the
dendroidal models, based on categories of trees as the shapes for
$\infty$-operads, is closer to the diagramatic operadic intuition, but
generally requires more work to make any construction. For exemple,
cocartesian fibrations and their straightening were studied
in~\cite{heuts11:_algeb} in the model of dendroidal sets, a point-set
model which is not manifestly model-invariant.

A more homotopically robust model is that of Segal dendroidal spaces,
or more generally Segal dendroidal objects in any complete
$(\infty,1)$-category. This model makes it clear that all that is
needed to model algebraic objects such as operads is a general
category of ``shapes'' and their Segal decompositions. This philosophy
was realised in~\cite{chu21:_homot_segal} which defines a notion of
algebraic pattern, a base shape category over which can be defined
Segal presheaves.

Thus one would like to define a notion of cocartesian operations and
fibrations for Segal objects over algebraic patterns, and in addition
construct the free fibration generated by an arbitrary morphism. Here,
we will not achieve this generality, but take a simpler route to
defining opfibrations. In~\cite{hermida04:_fibrat}, opfibrations of
multicategories were characterised as the morphisms whose image under
the functor taking monoidal envelopes is an opfibration (or
cocartesian fibration) of underlying categories. We will follow this
idea, and focus on constructing an envelope $\infty$-functor and
exploiting it to define cocartesian fibrations.

\subsection*{The direction problem and the plus construction}

To perform the above requires a notion of direction for the operations
of Segal objects: in categories, there are two directions for
cartesianity, from the source or from the target (giving rise to
cartesian and cocartesian morphisms), while for operads or their
``many-objects'' generalisation, virtual double categories, the notion
of (co)cartesian morphisms which has so far been explored looks from
the target (as in~\cite[Definition 5.1]{cruttwell10}), but we expect
that any choice of input to separate may give a notion of
cartesianity. Not all choices of direction, however, are as good as
others: for operads, the choice of the output to keep apart leads to
functoriality for composition, while the other choices do not (due to
the absence of a duality operation as for categories). As a
consequence, one may need to restrict the study to ``good''
orientations; this seems likely to be impossible to find for modular
operads.

In this work, we have elected to eschew the problem by replacing a
perhaps more canonical definition of directibility by a more practical
one. The presence of a notion of direction for operations means that
we can think of any composite of them as having an order of
progression. This is what is captured by the objects of the simplex
category $\bbDelta$, chains of morphisms going from a beginning to an
end. Hence we will base our notion of direction on this category,
declaring an algebraic pattern to be \emph{well-directed} if it can be
written as the output of a certain construction involving
$\bbDelta$. The appropriate construction to consider turns out to be a
variant of the plus construction suggested by Baez--Dolan and studied
by, among others, \cite{barwick18:_from} and~\cite{berger22:_momen}.

The plus construction of a pattern $\cat{O}$ is characterised by the
fact that its Segal objects are the weak Segal $\cat{O}$-fibrations,
or ``$\cat{O}$-operads''. In this language, envelopes for weak Segal
fibrations of patterns have been recently constructed
in~\cite{barkan22:_envel_algeb_patter}, building on the construction
of~\cite{lurie17:_higher_algeb} and the properties exhibited
in~\cite{haugseng21}. While in a context equivalent to ours, this
construction has orthogonal goals to the one performed here, as we
wish to always remain in the explicit setting of Segal objects.

\subsection*{Idea of the construction}
\label{sec:idea-construction}

In order to understand the idea of our construction of the monoidal
envelope functor, we describe it here in the case of the pattern
$\op{\bbGamma}$, whose Segal objects are the commutative monoids and
whose plus construction parameterises operads.

Let $\mathcal{O}$ be an operad. Its monoidal envelope is constructed
as a monoidal category $\func{Env}(\mathcal{O})$ which has as set of
objects the free monoid generated by the colours of $\mathcal{O}$,
whose elements are denoted as $C_{1}\otimes\cdots\otimes C_{n}$ or
simply $C_{1}\cdots C_{n}$. If $C_{1}\otimes\cdots\otimes C_{n}$ is
such a string of colours of $\mathcal{O}$ and $D$ is one colour, a
morphism $C_{1}\otimes\cdots\otimes C_{n}\to D$ is given by a
multimorphism $C_{1},\dots,C_{n}\to D$ in $\mathcal{O}$. If
$C_{1}\otimes\cdots\otimes C_{n}$ and
$D_{1}\otimes\cdots\otimes D_{m}$ are two such strings of colours of
$\mathcal{O}$, to defines a morphism
$C_{1}\otimes\cdots\otimes C_{n}\to D_{1}\otimes\cdots\otimes D_{m}$
one needs to further select of partition of the inputs
$(C_{1},\dots, C_{n})$ into $m$ (possibly empty) parts.

The above is so far just a description of the underlying category
$\cat{Env}(\mathcal{O})$ of the envelope of $\mathcal{O}$; to define
it as a monoidal category, or representably monoidal operad, one must
also define multimorphisms of higher arity in the operadic
structure. Let
$(C^{i}_{1}\otimes\cdots\otimes C^{i}_{p_{i}})_{i\in\bbracks{1,n}}$ be
$n$ colours of $\func{Env}(\mathcal{O})$ and let
$D_{1}\otimes\cdots\otimes D_{m}$ be a further colour. By the
representability condition, a multimorphism
$(C^{1}_{1}\cdots C^{1}_{p_{1}},\dots,C^{n}_{1}\cdots
C^{n}_{p_{n}})\to D_{1}\otimes\cdots\otimes D_{m}$ is given by a
morphism
$C^{1}_{1}\otimes\cdots\otimes C^{n}_{p_{n}}\to
D_{1}\otimes\cdots\otimes D_{m}$, that is a partition of the entries
and a collection of multimorphisms to each $D_{i}$.

In the dendroidal model, one simply defines the object
$\func{Env}(\mathcal{O})(\corol{n})$ of all $n$-ary morphisms without
specifying their sources and targets. To describe this, it becomes
useful to reverse the thinking: taking a family of multimorphisms of
$\mathcal{O}$, we can ask how to interpret it as a multimorphism in
$\func{Env}(\mathcal{O})$. If $C_{1},\cdots,C_{r}$ is the union of the
domains of the multimorphisms in the family considered, the
decomposition in family provides a partition of $p$ indexed by the
targets; however, from the point of view of $\func{Env}(\mathcal{O})$,
this partition is completely artificial as it is only used to
construct a morphism whose target may consist of several colours. Thus
it must be forgotten. Meanwhile, if the family is to be interpreted as
a multimorphism of specified arity $n$ in $\func{Env}(\mathcal{O})$,
the set of colours $(C_{i})_{i\in\bbracks{1,r}}$ must be endowed with
a partition into $n$ parts.

In a formula, we have that
\begin{equation}
  \label{eq:2}
  \func{Env}(\mathcal{O})(\corol{n})
  =\coprod_{\substack{r\in\mathbb{N}\\\lambda\text{ partition of $r$
        in $n$}}}\prod_{i=1}^{m}\mathcal{O}(\corol{\lambda(i)})\text{.}
\end{equation}
Recall that partitions are the same thing as active morphisms in
$\op{\bbGamma}$ (which is where the power of the plus construction,
relating $\op{\Gamma}$ to the pattern for operads, comes into
play). We can then interpret the coproduct on partitions as a colimit
over morphisms in $(\op{\bbGamma})^{\mathrm{act}}$.

Viewed in this way, this formula is very reminiscent to the one
computing (pointwise) oplax extensions\footnote{Oft known as left Kan
  extensions}, with one difference: the colimit is taken only of the
trees whose height is that of a corolla, \emph{i.e.}  compatibly with
the projection to $\bbDelta$. To that end, the notion of extension
will need to be refined to a ``fibrewise'' one.

\subsection*{Outline of the paper}

In~\cref{sec:plus-constr-Seg-obj}, we define the plus construction for
appropriate algebraic patterns, and establish its main properties, in
a manner very similar to the study of the plus construction of
${\op{\bbGamma}}^{\natural}$ carried out in~\cite{chu18:_two}
or~\cite{chu20:_enric}. Then, in~\cref{sec:mono-envel-groth}, we will
use it to construct the ``representably monoidal'' envelope of a Segal
object. While the construction makes sense for general well-directed
patterns, we are only able to exhibit its good monoidal properties by
restricting patterns such as the one ${\op{\bbOmega}}^{\flat}$ for
operads. Finally, in~\cref{sec:cart-mono-struct-appl} we further use
this envelope functor to define a straightening/unstraightening
Grothendieck construction for $\infty$-operads; the results of this
section are not new and are simply variants of results appearing
in~\cite{lurie17:_higher_algeb}, the combinatorics of whose proofs are
in fact directly reused, but this language provides a new, more
operadic, point of view on them.

\subsection*{Acknowledgements}
\label{sec:acknowledgements}

The core results of this note were worked out as part of my PhD with
Étienne Mann, and I thank him for many useful discussions.

I acknowledge funding from the grant of the Agence Nationale de la
Recherche ``Categorification in Algebraic Geometry'' ANR-17-CE40-0014
and the European Research Council (ERC) under the European Union’s
Horizon 2020 research and innovation programme (Grant ``Derived
Symplectic Geometry and Applications'' Agreement No. 768679)

\section{Review of the language of algebraic patterns}
\label{sec:revi-plus-constr}

\subsection{Algebraic patterns and Segal objects}
\label{sec:lang-algebr-patt}

\begin{definition}[Algebraic pattern]
  \label{def:alg-pattrn}
  An \textbf{algebraic pattern} is an $(\infty,1)$-category endowed
  with a unique factorisation system and a selected class of objects called
  \textbf{elementary}. The morphisms in the left class of the
  factorisation system are called \textbf{inert}, and those in the
  right class \textbf{active}.
\end{definition}

\begin{notation}
  Inert maps in an algebraic pattern are usually denoted as
  $\rightarrowtail$, while active maps are denoted as
  $\rightsquigarrow$.
  
  If $\cat{O}$ is the underlying $(\infty,1)$-category of the
  algebraic pattern considered, one writes $\cat{O}^{\mathrm{inrt}}$
  and $\cat{O}^{\mathrm{act}}$ for the wide and locally full
  sub-$(\infty,1)$-categories whose morphisms are respectively the
  inert and the active morphisms.

  We also denote $\cat{O}^{\mathrm{el}}$ the full
  sub-$(\infty,1)$-category of $\cat{O}^{\mathrm{inrt}}$ on the
  selected elementary objects. For any object $O\in\cat{O}$, we
  further write
  $\cat{O}^{\mathrm{el}}_{O/}\coloneqq
  \cat{O}^{\mathrm{el}}\times_{\cat{O}^{\mathrm{inrt}}}\cat{O}^{\mathrm{inrt}}_{O/}$,
  the $(\infty,1)$-category of inert morphisms from $O$ to an
  elementary object.
\end{notation}

\begin{exmp}
  Segal's category $\bbGamma$, (a skeleton of) the opposite of the
  category $\cat{FinSet}_{\ast}$ of pointed finite sets, admits an
  active-inert factorisation system where a morphism of pointed finite
  sets $f\colon(S,s_{0})\to(T,t_{0})$ is \textbf{inert} if for every
  $t\in T\setminus\{t_{0}\}$, the preimage $f^{-1}(t)$ consists of
  exactly one element, and \textbf{active} if the only element of $S$
  mapped to $t_{0}$ is $s_{0}$ --- so that, in particular, the
  subcategory of active morphisms ${\op{\bbGamma}}^{\mathrm{act}}$
  identifies with the category of finite sets (by mapping $(S,s_{0})$
  to $S\setminus\{s_{0}\}$).
  
  The induced inert-active factorisation system on $\op{\bbGamma}$
  gives rise to two algebraic pattern structures, from two choices of
  elementary objects. The algebraic pattern ${\op{\bbGamma}}^{\flat}$
  has as elementaries the two-element sets (isomorphic to
  $\langle1\rangle$), while the pattern ${\op{\bbGamma}}^{\natural}$
  has as elementaries the singletons (isomorphic to $\langle0\rangle$)
  and the two-element sets.
\end{exmp}

\begin{exmp}
  The (non-augmented) simplicial indexing category $\bbDelta$,
  identified with the category of ordered non-empty finite sets and
  order-preserving maps between them, admits a factorisation system in
  which a map is \textbf{active} if it preserves the top and bottom
  elements, and \textbf{inert} if it corresponds to the inclusion of a
  linear subset.
  
  The induced inert-active factorisation system on $\op{\bbDelta}$
  gives rise to two algebraic pattern structures: the algebraic
  pattern ${\op{\bbDelta}}^{\flat}$ has as elementaries the
  two-element sets (isomorphic to $[1]$), while the pattern
  ${\op{\bbGamma}}^{\natural}$ has as elementaries the singletons
  (isomorphic to $[0]$) and the two-element sets.
\end{exmp}

\begin{exmp}
  Of particular interest to us, the dendroidal category $\bbOmega$ of
  Moerdijk--Weiss can be described as a category of (non-planar)
  rooted trees, with morphisms the morphisms of free (coloured,
  symmetric) operads generated by these trees. We shall give an
  alternate construction of (a sufficient subcategory of) it
  in~\cref{sec:plus-constr-Seg-obj}.

  Certain particularly interesting trees can be distinguished:
  \begin{description}
  \item[the free-living edge] is the tree, denoted $\edge$, consisting
    of one edge but no vertex (so the operad it freely generates has
    one colour, and only its identity unary morphism);
  \item[the corollas] are the trees, denoted $\corol{n}$ with
    $n\in\mathbb{N}$, with a single vertex and $n+1$ edges (one of
    which is the root) attached to it. Note that the corolla
    $\corol{n}$ with $n$ leaves has as automorphism group the
    symmetric group $\mathfrak{S}_{n}$.
  \end{description}

  The category $\bbOmega$ admits a factorisation system in which a
  morphism is \textbf{active} if it is boundary-preserving, and
  \textbf{inert} if it corresponds to a subtree inclusion (which, in
  particular, is valence-preserving on the vertices).

  The induced inert-active factorisation system on $\op{\bbOmega}$
  gives again rise to two algebraic pattern structures: the algebraic
  pattern ${\op{\Omega}}^{\flat}$ has as elementaries the corollas
  $\corol{n}$, while the pattern ${\op{\bbOmega}}^{\natural}$ has as
  elementaries the corollas and the free-living edge $\edge$.
\end{exmp}

\begin{definition}[Morphisms of algebraic patterns]
  \label{def:morph-alg-pattrns}
  Let $\cat{O}$ and $\cat{P}$ be algebraic patterns. A
  \textbf{morphism of algebraic patterns} from $\cat{O}$ to $\cat{P}$
  is an $(\infty,1)$-functor $\cat{O}\to\cat{P}$ which preserves
  active and inert morphisms and elementary objects.
\end{definition}

\begin{exmp}
  There is a functor $\bbDelta\to\bbOmega$ which, taking the standard
  skeleton of $\bbDelta$, sends $[n]$ to the linear tree with $n$
  nodes and $n+1$ edges. One may remark that it is fully faitful, and
  can also be identified with the canonical projection functor
  $\bbOmega_{/\edge}\to\bbOmega$. By translating the definition of
  inert and active morphisms for the given factorisation system on
  $\op{\bbDelta}$, one immediately sees that this functor induces
  morphisms of algebraic patterns
  ${\op{\bbDelta}}^{\natural}\to{\op{\bbOmega}}^{\natural}$ and
  ${\op{\bbDelta}}^{\flat}\to{\op{\bbOmega}}^{\flat}$.
\end{exmp}

\begin{definition}[Segal objects]
  Let $\cat{O}$ be an algebraic pattern. An $(\infty,1)$-category
  $\cat{C}$ is said to be \textbf{$\cat{O}$-complete} if it admits
  limits of diagrams with shape $\cat{O}^{\mathrm{el}}_{O/}$ for any
  $O\in\cat{O}$.

  Let $\cat{C}$ be an $\cat{O}$-complete $(\infty,1)$-category. A
  \textbf{Segal $\cat{O}$-object} in $\cat{C}$ is an
  $(\infty,1)$-functor $\func{X}\colon\cat{O}\to\cat{C}$ such that
  $\restr{\func{X}}_{\cat{O}^{\mathrm{inrt}}}$ is a lax extension of
  $\restr{\func{X}}_{\cat{O}^{\mathrm{el}}}$ (along the
  inclusion). Explicitly, this means that for any $O\in\cat{O}$ the
  canonical map
  \begin{equation}
    \label{eq:segal-cond-def-canon-map}
    \func{X}(O)\to\lim_{E\in\cat{O}^{\mathrm{el}}_{O/}}\func{X}(E)
  \end{equation}
  is invertible.

  The $(\infty,1)$-category $\seg{\cat{O}}(\cat{C})$ of Segal
  $\cat{O}$-objects in $\cat{C}$ is the full sub-$(\infty,1)$-category
  of $\funcs{O}{C}$ spanned by the Segal objects.
\end{definition}

\begin{exmp}
  \begin{itemize}
  \item A Segal ${\op{\bbGamma}}^{\flat}$-object is a commutative
    algebra object (or $\mathcal{E}_{\infty}$-algebra object).
  \item A Segal ${\op{\bbDelta}}^{\flat}$-object is an associative
    algebra object (or $\mathcal{A}_{\infty}$-algebra object, or
    $\mathcal{E}_{1}$-algebra), while a Segal
    ${\op{\bbDelta}}^{\natural}$-object is an internal category.
  \item A Segal ${\op{\bbOmega}}^{\flat}$-object is an internal
    monochromatic operad, while a Segal
    ${\op{\bbOmega}}^{\natural}$-object is an internal (coloured) operad.
  \end{itemize}
\end{exmp}

\begin{remark}
  Let $\cat{O}$ be an algebraic pattern such that the inclusion
  $\cat{O}^{\mathrm{el}}\hookrightarrow\cat{O}^{\mathrm{inrt}}$ is
  codense. Then for any $O\in\cat{O}$, the corepresentable
  $\yo_{\op{\cat{O}}}(O)\colon\cat{O}\to\infgrpds$ is a Segal
  $\cat{O}$-$\infty$-groupoid --- this is an immediate consequence of
  the limit-preservation property of hom $\infty$-functors. It should
  be viewed as the Segal object generated by $O$.
\end{remark}

\begin{exmp}
  Over $\op{\bbDelta}$, the Segal object $\yo{[n]}$ corresponds to the
  linear category $\mathbb{n+1}$ with $n$ successive arrows.
\end{exmp}

\begin{exmp}
  Over $\op{\bbOmega}$, the Segal object $\yo{\corol{n}}$ generated by
  the corolla with $n+1$ flags is also denoted $\corol{n}$ and called
  the corolla. It corresponds to the operad with $n+1$ colours
  $O_{1},\dots,O_{n+1}$ and, for each permutation
  $\sigma\in\mathfrak{S}_{n}$, a single operation of signature
  $(O_{\sigma(1)},\dots,O_{\sigma(n)};O_{n+1})$.
\end{exmp}

\subsection{Morphisms of algebraic patterns I: (weak) Segal
  fibrations}
\label{sec:morph-algebr-patt}

\begin{definition}[Segal morphism of algebraic patterns]
  A morphism of algebraic patterns $\func{F}\colon\cat{O}\to\cat{P}$
  is said to be a \textbf{Segal morphism} if ``it preserves Segal
  conditions'', that is if for any $\cat{P}$-complete
  $(\infty,1)$-category $\cat{C}$, the induced $(\infty,1)$-functor
  $\restr{\func{F}^{\ast}}_{\seg{\cat{P}}}\colon\seg{\cat{P}}(\cat{C})
  \subset\funcs{P}{C}\to\funcs{O}{C}$ factors through
  $\seg{\cat{O}}(\cat{C})\hookrightarrow\funcs{O}{C}$.
\end{definition}

\begin{remark}
  By~\cite[Lemma 4.5]{chu21:_homot_segal}, it is enough to check
  Segality of a morphism with $\cat{C}=\infgrpds$, that is to check
  preservation of Segal $\infty$-groupoids. Then the condition for
  $\func{F}$ to be a Segal morphism can be written in a formula as:
  for every $O\in\cat{O}$, for every Segal $\cat{P}$-$\infty$-groupoid
  $\func{X}$, the morphism of $\infty$-groupoids
  \begin{equation}
    \label{eq:segal-morph-def}
    \lim_{\cat{P}^{\mathrm{el}}_{\func{F}(O)/}}\func{X}\to
    \lim_{\cat{O}^{\mathrm{el}}_{O/}}\func{X}\circ\func{F}^{\mathrm{el}}
  \end{equation}
  induced by the $(\infty,1)$-functor
  $\cat{O}^{\mathrm{el}}_{O/}\to\cat{P}^{\mathrm{el}}_{\func{F}(O)/}$
  is an equivalence.
\end{remark}

\begin{construct}
  \label{construct:segal-mndl-cat-segal-fibrtn}
  Suppose $\func{V}\colon\cat{O}\to\infcats$ is an $\cat{O}$-monoidal
  $(\infty,1)$-category. Passing to the Grothendieck construction of
  the functor produces a cocartesian fibration
  $\grothco\func{V}\to\cat{O}$. We shall refer to a cocartesian
  fibration over $\cat{O}$ whose associated $(\infty,1)$-functor
  $\cat{O}\to\infcats$ satisfies the Segal conditions as a
  \textbf{Segal fibration} over $\cat{O}$.

  By~\cite[Proposition 2.1.2.5]{lurie17:_higher_algeb}, if
  $\grothco\func{V}\to\cat{O}$ is a Segal $\cat{O}$-fibration, the
  inert-active factorisation system on $\cat{O}$ lifts to one on
  $\grothco\func{V}$, endowing it with a structure of algebraic
  pattern. If $\cat{P}\to\cat{O}$ is another algebraic pattern over
  $\cat{O}$, we shall define a \textbf{$\cat{P}$-algebra} in
  $\grothco\func{V}$ to be a morphism of algbraic pattern over
  $\cat{O}$ from $\cat{P}$ to $\grothco\func{V}$.
\end{construct}

\begin{definition}[Weak Segal fibration]
  \label{def:wk-segal-fibrtn}
  Let $\cat{O}$ be an algebraic pattern. A \textbf{weak Segal
    $\cat{O}$-fibration} (also called $\cat{O}$-operad) is an
  $(\infty,1)$-functor $\func{p}\colon\cat{X}\to\cat{O}$ such that:
  \begin{enumerate}
  \item for every object $X\in\cat{X}$, every inert arrow
    $i\colon\func{p}X\to O$ in $\cat{O}$ admits a
    $\func{p}$-cocartesian lift $i_{!}\colon X\to i_{!}X$;
  \item for every object $O\in\cat{O}$, the $(\infty,1)$-functor
    $\cat{X}_{O}\to\lim_{E\in\cat{O}^{\mathrm{el}}_{O/}}\cat{X}_{E}$
    induced by the cocartesian morphisms over inert arrows is
    invertible;
  \item for every $X\in\cat{X}$ and every choice of
    $\func{p}$-cocartesian lift of the tautological diagram
    $\func{i}\colon\cat{O}^{\mathrm{el}}_{\func{p}X/}\to\cat{O}$ (of
    inert morphisms from $\func{p}X$) to an
    $\func{i}_{!}\colon(\cat{O}^{\mathrm{el}}_{\func{p}X/})^{\lhd}\to\cat{X}$
    taking the cone point to $X$, for every $Y\in\cat{X}$, the
    commutative square
    \begin{equation}
      \label{eq:sqr-weak-segl-fibr-pattrn}
      \begin{tikzcd}
        \cat{X}(Y,X) \arrow[r] \arrow[d] &
        \lim_{E\in\cat{O}^{\mathrm{el}}_{\func{p}X/}}\cat{X}(Y,\func{i}_{!}(E))
        \arrow[d] \\
        \cat{O}(\func{p}Y,\func{p}X) \arrow[r] &
        \lim_{E\in\cat{O}^{\mathrm{el}}_{\func{p}X/}}\cat{O}(\func{p}Y,\func{i}(E)=E)
      \end{tikzcd}
    \end{equation}
    is cartesian.
  \end{enumerate}
\end{definition}

\begin{exmp}
  \begin{itemize}
  \item A weak Segal ${\op{\bbGamma}}^{\flat}$-fibration is the
    $(\infty,1)$-category of operators of an $\infty$-operad in the
    sense of~\cite[Definition 2.1.1.10]{lurie17:_higher_algeb}, while
    a weak Segal ${\op{\bbGamma}}^{\natural}$-fibration is a
    generalised $\infty$-operad in the sense of~\cite[Definition
    2.3.2.1]{lurie17:_higher_algeb}.
  \item A weak Segal ${\op{\bbDelta}}^{\flat}$-fibration is the
    $(\infty,1)$-category of operators of a non-symmetric
    $\infty$-operad in the sense of~\cite[Definition 2.2.6, Definition
    3.1.3]{gepner15:_enric} while a weak Segal
    ${\op{\bbDelta}}^{\natural}$-fibration is virtual double
    $\infty$-category, or generalised $\infty$-operad
    in~\cite[Definition 2.4.1, Definition 3.1.13]{gepner15:_enric}.
  \end{itemize}
\end{exmp}

\begin{definition}[Morphisms of weak Segal fibrations]
  By~\cite{chu21:_homot_segal}, the source $\cat{X}$ of a weak Segal
  $\cat{O}$-fibration $\func{p}\colon\cat{X}\to\cat{O}$ inherits an
  algebraic pattern structure in which active morphisms are those
  lying over an active morphism in $\cat{O}$, inert morphisms are the
  $\func{p}$-cocartesian morphisms lying over inert arrows of
  $\func{O}$, and elementaries are the objects lying over elementary
  objects.

  The $(\infty,2)$-category of weak Segal $\cat{O}$-fibrations
  $\wkfibs{\cat{O}}$ is the locally full sub-$(\infty,2)$-category of
  $\infcats_{/\cat{O}}$ spanned by the weak Segal $\cat{O}$-fibrations
  and Segal morphisms thereof.
\end{definition}

It can be checked directly that any Segal fibration as
in~\cref{construct:segal-mndl-cat-segal-fibrtn} is in particular a
weak Segal fibration. In fact Segal $\cat{O}$-fibrations are exactly
those $(\infty,1)$-functors to $\cat{O}$ which are both weak Segal
fibrations and cocartesian fibrations. This provides a (non-full)
inclusion $(\infty,2)$-functor from the $(\infty,2)$-category of Segal
fibrations into that of weak Segal fibrations.

\subsection{Morphisms of algebraic patterns II: Combinatorics of
  enrichable structures}
\label{sec:morph-algebr-patt-2}

To finish this preliminary~\namecref{sec:revi-plus-constr}, we state
(a variant of) some definitions\footnote{which were presented in the
  seminar talk available at~\url{https://www.msri.org/seminars/25057}}
of the forthcoming paper~\cite{chu22:_enric}.

\begin{definition}[Combinatorial structure]
  A \textbf{combinatorial pattern} is an algebraic pattern $\cat{O}$
  equipped with a morphism
  $\uly{-}\colon\cat{O}\to{\op{\bbGamma}}^{\natural}$.

  When $(\cat{O},\uly{-})$ is a combinatorial pattern, we let
  $\cat{O}^{\flat}\to{\op{\bbGamma}}^{\flat}$ denote the base-change
  \begin{equation}
    \label{eq:combin-flat-alteration}
    \cat{O}^{\flat}\coloneqq
    \cat{O}\times_{{\op{\bbGamma}}^{\natural}}{\op{\bbGamma}}^{\flat}\text{.}
  \end{equation}
\end{definition}

Thanks to~\cite[Corollary 5.5]{chu21:_homot_segal}, we see that the
fibre product appearing in~\cref{eq:combin-flat-alteration} has a very
simple description: it consists of the category $\cat{O}$ equipped
with its same factorisation system, and the choice of only those
elementary objects living over $\langle1\rangle\in\op{\bbGamma}$
(\emph{i.e.}  excluding those over $\langle0\rangle$).

\begin{definition}[Enrichable pattern]
  A combinatorial structure $\uly{-}$ on an algebraic pattern
  $\cat{O}$ is \textbf{enrichable} if, for any $O\in\cat{O}$, the
  morphism
  $\cat{O}^{\flat,\mathrm{el}}_{O/}
  \to{\op{\bbGamma}}^{\flat,\mathrm{el}}_{\uly{O}/}$
  is an equivalence.
\end{definition}

\begin{remark}
  \label{remark:cartsn-pattrn-inrt-list}
  Viewing $\op{\bbGamma}$ as (the standard skeleton of) the category
  of pointed finite sets, one verifies that any
  $\langle n\rangle\in\op{\bbGamma}$ admits exactly $n$ inert
  morphisms to the unique elementary $\langle1\rangle$, the pointed
  morphisms $\rho_{i}\colon\langle n\rangle\to\langle1\rangle$ mapping
  $i$ to $1$ and every other element of $\langle n\rangle$ to the
  base-point. The condition of being an enrichable pattern then means
  that, for any $O\in\cat{O}$, the $(\infty,1)$-category
  $\cat{O}^{\flat,\mathrm{el}}_{O/}$ must be equivalent to the
  discrete set of the (essentially unique) lifts $\rho_{i,!}$ of the
  $\rho_{i}$. In particular, the $\cat{O}^{\flat}$-Segal condition for
  a precosheaf $\func{X}$ on $\cat{O}$ is constrained to being the
  finite product condition
  \begin{equation}
    \label{eq:segal-condtion-prod-cartesian-pattrn}
    \func{X}(O)\simeq\prod_{i=1}^{\uly{O}}\func{X}(\rho_{i,!}O)\text{.}
  \end{equation}
\end{remark}

\begin{exmp}
  The functor $\ast\to\op{\bbGamma}$ picking the object
  $\langle1\rangle$ defines a structure of enrichable pattern on the
  terminal algebraic pattern.
\end{exmp}

\begin{exmp}
  There is a functor $\op{\bbDelta}\to\op{\bbGamma}$ mapping $[n]$ to
  $\langle n\rangle$ and sending an arrow of $\op{\bbDelta}$
  corresponding to $\phi\colon[n]\to[m]$ in $\bbDelta$ to
  $\uly{\phi}\colon\langle m\rangle\to\langle n\rangle$ given by
  \begin{equation}
    \label{eq:cart-pattrn-delta}
    \uly{\phi}(i)=
    \begin{cases}
      j & \text{if $\phi(j-1)<i\leq\phi(j)$}\\
      \ast & \text{otherwise.}
    \end{cases}
  \end{equation}
  It can be checked directly that this is a structure of enrichable
  pattern on ${\op{\bbDelta}}^{\natural}$.
\end{exmp}

\begin{exmp}
  A functor $\op{\bbOmega}\to\op{\bbGamma}\simeq\cat{FinSet}_{\ast}$
  is defined in~\cite[Definition 4.1.16]{chu20:_enric} in the
  following way. A tree $T$ with set of vertices $V(T)$ is mapped to
  the freely pointed set $V(T)_{+}$. A morphism
  $T^{\prime}\leftarrow T$ in $\op{\bbOmega}$ is mapped to the pointed
  morphism $V(T^{\prime})_{+}\to V(T)_{+}$ which sends a vertex
  $v\in V(T^{\prime})$ to the unique vertex of $T$ whose image subtree
  contains $v$, or to the basepoint if there is no such vertex. By
  [loc. cit., Lemma 4.1.18], this functor preserves the inert-active
  factorisation system, and thus defines a morphism of algebraic
  patterns. One checks by inspection that it is an enrichable
  structure.
  
  In~\cref{pros:enrich-plus-constr} we will provide another
  construction of this enrichable structure from the point of view
  of $\bbOmega$ as obtained from the plus construction.
\end{exmp}

We recall from~\cite{lurie17:_higher_algeb} the definition of
semi-inert arrows in $\op{\bbGamma}$. A map of pointed finite sets
$f\colon(S,s_{0})\to(T,t_{0})$ is \textbf{semi-inert} if for any
$t\in T\setminus\{t_{0}\}$, there is at most one element in
$f^{-1}(t)$.

\begin{definition}
  Let $(\cat{O},\uly{\cdot})$ be a combinatorial algebraic pattern. An
  arrow $f$ of $\cat{O}$ is \textbf{semi-inert} if $\uly{f}$ is a
  semi-inert arrow of $\op{\bbGamma}$.
\end{definition}

\begin{exmp}
  In ${\op{\bbDelta}}^{\natural}$, the semi-inert morphisms are those
  corresponding to the cellular morphisms of $\bbDelta$ as defined
  in~\cite{haugseng17:_morit} and~\cite{haugseng16:_bimod}, that is
  maps of totally ordered sets $f\colon S\to S^{\prime}$ such that for
  all $s\in S$,
  $f(\operatorname{succ}s)\leq\operatorname{succ}(f(s))$.

  For the product pattern $({\op{\bbDelta}}^{\natural})^{n}$ equipped
  with
  $({\op{\bbDelta}}^{\natural})^{n}
  \xrightarrow{\uly{-}}({\op{\bbGamma}}^{\natural})^{n}
  \xrightarrow{\wedge}{\op{\bbGamma}}^{\natural}$, we also recover the
cellular morphisms of~\cite{haugseng17:_morit}, those maps all of
whose components are cellular in $\op{\bbDelta}$.
\end{exmp}

\begin{remark}
  As observed in~\cite{haugseng17:_morit}, an arrow
  $f\colon O\to O^{\prime}$ of $\cat{O}$ is semi-inert if and only if,
  for any elementary $E$ and any inert morphism
  $k\colon O^{\prime}\rightarrowtail E$, the composite
  $O\xrightarrow{k\circ f}E$ is semi-inert.
\end{remark}

\begin{exmp}
  \label{exmp:semi-inert-lifts-el}
  In $\op{\bbGamma}$, any object $\langle n\rangle$ admits precisely
  $n$ active morphisms from the unique elementary object
  $\langle1\rangle$, the functions
  $\lambda_{i}\colon\langle1\rangle \rightsquigarrow\langle n\rangle$
  sending $1\in\langle1\rangle$ to $i\in\langle n\rangle$ for any
  $1\leq i\leq n$. They are semi-inert. Likewise, the unique map
  $\langle0\rangle\to\langle n\rangle$ is also semi-inert.

  It follows that, in any combinatorial pattern, any active arrow from
  an elementary object is semi-inert.
\end{exmp}

\begin{definition}[Momentous pattern]
  A combinatorial algebraic pattern $(\cat{O},\uly{-})$ is
  \textbf{momentous} if any active morphism from a $\flat$-elementary
  admits an essentially unique inert retraction.
\end{definition}

Momentous patterns which are enrichable are closely related to (an
$\infty$-categorical version of) the hypermoment categories
of~\cite{berger22:_momen}.

\begin{exmp}
  \label{exmp:gamma-act-list-retrac}
  The pattern ${\op{\bbGamma}}^{\natural}$ itself is momentous: for
  any object $\langle n\rangle$, the function
  $\lambda_{i}\colon\langle1\rangle\to\langle n\rangle$
  of~\cref{exmp:semi-inert-lifts-el} admits as unique inert retraction
  $\rho_{i}$. In fact $\lambda_{i}$ is also the unique active section
  of $\rho_{i}$, so there is a one-to-one correspondence between inert
  morphisms to $\flat$-elementaries and (automaticaly semi-inert)
  active morphisms from $\flat$-elementaries.

  As a consequence, for any momentous enrichable pattern $\cat{O}$ and
  any object $O\in\cat{O}$, the (semi-inert) active morphisms from
  $\flat$-elementary objects to $O$ can be identified with a subset of
  $\cat{O}^{\mathrm{el}}_{O/}
  \simeq{\op{\bbGamma}}^{\natural,\mathrm{el}}_{\uly{O}/}$.
\end{exmp}

\section{The plus construction and its Segal objects}
\label{sec:plus-constr-Seg-obj}

\subsection{Categories of patterned trees}
\label{sec:categ-patt-trees}

The following construction is a variant of one due
to~\cite{barwick18:_from} in the setting of operator categories,
inspired by the plus construction or ``slice operads''
of~\cite{baez98:_higher_dimen_algeb_iii}, and which was also used
in~\cite{chu18:_two} and studied in~\cite{berger22:_momen} in the
setting of hypermoment categories --- as well
as~\cite{batanin21:_operad_koszul}, in a different form, for operadic
categories.

\begin{construct}
  Let $\cat{O}$ be a combinatorial algebraic pattern. Consider the
  restriction
  $\restr{\underline{\yo}_{\infcats}(\cat{O}^{\mathrm{act}})}_{\bbDelta}
  \colon[n]\mapsto\funcs{[\mathnormal{n}]}{\cat{O}^{\mathrm{act}}}$ to
  $\bbDelta\subset\infcats$ of the $(\infty,2)$-functor represented by
  $\cat{O}^{\mathrm{act}}$ and let
  $\bbDelta_{\cat{O}}^{\mathrm{pre}}\to\bbDelta$ be its Grothendieck
  construction. Note that $\bbDelta_{\cat{O}}^{\mathrm{pre}}$ is
  equivalent to $\bbDelta_{/\cat{O}^{\mathrm{act}}}$ --- where
  $\bbDelta$ is seen as a subcategory of $\infcats$ --- or also to
  $\bbDelta_{/\nerve\cat{O}^{\mathrm{act}}}$, where
  $\nerve\colon\infcats\hookrightarrow\prshvs{\bbDelta}$ is a nerve
  functor (for example incarnating $(\infty,1)$-categories as complete
  Segal $\infty$-groupoids, or as quasicategories, to the reader's
  preference). Thus an
  \begin{description}
  \item[object] of $\bbDelta_{\cat{O}}^{\mathrm{pre}}$ consists of a
    pair $([n],O_{\bullet})$ where $[n]\in\bbDelta$ and
    $O_{\bullet}\colon[n]\to\cat{O}^{\mathrm{act}}$ is a linear
    diagram in $\cat{O}^{\mathrm{act}}$, that is a sequence
    $O_{0}\rightsquigarrow O_{1}\rightsquigarrow\dots\rightsquigarrow
    O_{n}$ of active morphisms in $\cat{O}$, while a
  \item[morphism]
    $([n],O_{\bullet})\to([n^{\prime}],O^{\prime}_{\bullet})$ consists
    of a pair $(\phi,f_{\bullet})$ where
    $\phi\colon[n]\to[n^{\prime}]$ is a map in $\bbDelta$ and
    $f_{\bullet}\colon O_{\bullet}\Rightarrow
    O^{\prime}_{\phi(\bullet)}= O^{\prime}_{\bullet}\circ\phi$ is a
    natural transformation of $[n]$-shaped diagrams in
    $\cat{O}^{\mathrm{act}}$.
  \end{description}

  We define $\bbDelta_{\cat{O}}$ to be the wide and locally full
  (\emph{i.e.} containing all objects and all higher morphisms between
  a selection of $1$-morphisms) sub-$(\infty,1)$-category of
  $\bbDelta_{\cat{O}}^{\mathrm{pre}}$ on those morphisms
  $(\phi,f_{\bullet})$ such that
  \begin{itemize}
  \item $f_{\bullet}$ is component-wise semi-inert (in addition to
    active), that is each
    $f_{i}\colon O_{i}\rightsquigarrow O^{\prime}_{\phi(i)}$, for
    $i\in[n]$, is semi-inert, and
  \item $f_{\bullet}$ is a cartesian (or equifibred) natural
    transformation, that is for any morphism $i<j$ in $[n]$ the
    naturality square
    \begin{equation}
      \label{eq:nat-square-cart-trsfm-delta-pattrn}
      \begin{tikzcd}
        O_{i} \arrow[r,squiggly,"f_{i}"] \arrow[d,squiggly,"O_{i<j}"']
        &
        O^{\prime}_{\phi(i)} \arrow[d,squiggly,"O^{\prime}_{\phi(i<j)}"] \\
        O_{j} \arrow[r,squiggly,"f_{j}"'] & O^{\prime}_{\phi(j)}
      \end{tikzcd}
    \end{equation}
    is cartesian in $\cat{O}^{\mathrm{act}}$.
  \end{itemize}
\end{construct}

\begin{remark}
  \label{remark:morph-plus-constr-minl-data}
  Given fixed $\phi\colon[m]\to[n]$ in $\bbDelta$ and
  $([n],P_{\bullet})$ in $\bbDelta_{\cat{O}}$ over $[n]$, morphisms
  $(\phi,f_{\bullet})\colon([m],O_{\bullet})\to([n],P_{\bullet})$ in
  $\bbDelta_{\cat{O}}$ lifting $\phi$ are essentially determined, if
  they exist, by their underlying arrow
  $f_{m}\colon O_{m}\to P_{\phi(m)}$ at the terminal object
  $m\in[m]$. Indeed, for each $i\in[m]$ the object $O_{i}$ and arrow
  $f_{i}$ are required by the pullback condition
  in~\cref{eq:nat-square-cart-trsfm-delta-pattrn} to be the
  base-change of $O_{m}$ and $f_{m}$ along $i\leq m$.
\end{remark}

\begin{definition}
  Let $\cat{O}$ be a combinatorial algebraic pattern. The
  $(\infty,1)$-category $\bbDelta_{\cat{O}}$ is called the
  $(\infty,1)$-category of \textbf{$\cat{O}$-forests}.

  The full sub-$(\infty,1)$-category $\bbDelta^{(1)}_{\cat{O}}$ of
  $\bbDelta_{\cat{O}}$ on the $([n],O_{\bullet})$ such that
  $O_{n}\in\cat{O}^{\flat,\mathrm{el}}$ is called the
  $(\infty,1)$-category of \textbf{misodendric} $\cat{O}$-forests, or
  \textbf{$\cat{O}$-trees}.
\end{definition}

\begin{exmp}
  For the terminal algebraic pattern $\ast$, the
  $(\infty,1)$-categories of $\ast$-forests and of $\ast$-trees both
  recover the simplex category $\bbDelta$.
\end{exmp}

\begin{exmp}
  \label{exmp:plus-construct-gamma-trees}
  For ${\op{\bbGamma}}^{\flat}$, we obtain a category of forests,
  \emph{i.e.} disjoint unions of trees, with level structures, whose
  misodendric subcategory is identified in~\cite{berger22:_momen} with a
  full subcategory of $\bbOmega$. Indeed, recalling that the active
  subcategory of $\op{\bbGamma}$ is equivalent to the category of
  finite sets, the object $O_{n}$ is to be thought of as the set of
  roots of the forest, and each $O_{i}$ is the set of leaves at level
  $n-i$. The morphisms in $(\op{\bbGamma})^{\mathrm{act}}$ give
  partitions of the leaves at levels $\ell$ corresponding to the node
  (recognised by its unique output leaf at level $\ell+1$) to which
  they lead.
\end{exmp}

\begin{exmp}
  For ${\op{\bbDelta}}^{\flat}$, we similarly obtain a category of
  planar trees (or rather, forests thereof) with level structures.
\end{exmp}

\subsection{The pattern structure on the plus construction}
\label{sec:pattern-bbdelta_cato}

\begin{lemma}
  The $\infty$-functor
  $\func{d}_{\cat{O}}\colon\bbDelta_{\cat{O}}\subset
  \bbDelta^{\mathrm{pre}}_{\cat{O}}\to\bbDelta$ is a cartesian
  fibration.
\end{lemma}

\begin{proof}
  We first observe as in~\cite{chu18:_two} that, due to the pullback
  condition in~\cref{eq:nat-square-cart-trsfm-delta-pattrn}, a
  morphism
  $(\phi,f_{\bullet})\colon([m],O_{\bullet})\to([n],P_{\bullet})$ in
  $\bbDelta_{\cat{O}}$ is $\func{d}_{\cat{O}}$-cartesian if and only
  if $f_{\bullet}$ is a natural equivalence. Indeed, to say that
  $(\phi,f_{\bullet})\colon([m],O_{\bullet})\to([n],P_{\bullet})$ is
  $\func{d}_{\cat{O}}$-cartesian is to say that
  \begin{equation}
    \label{eq:plus-construct-cart-fibr-def-square}
    \begin{tikzcd}[column sep=large]
      \bbDelta_{\cat{O},/([m],O_{\bullet})} \arrow[d]
      \arrow[r,"{(\phi,f_{\bullet})\circ-}"] &
      \bbDelta_{\cat{O},/([n],P_{\bullet})} \arrow[d] \\
      \bbDelta_{/[m]} \arrow[r,"\phi\circ-"'] & \bbDelta_{/[n]}
    \end{tikzcd}
  \end{equation}
  is a cartesian square.

  Let then $([k],L_{\bullet})$ be any object of $\bbDelta_{\cat{O}}$,
  and take the fibre of~\cref{eq:plus-construct-cart-fibr-def-square}
  at $([k],L_{\bullet})$, which we decompose as follows:
  \begin{equation}
    \label{eq:plus-construct-cart-fibr-test-square}
    \begin{tikzcd}
      \bbDelta_{\cat{O}}\bigl(([k],L_{\bullet}),([m],O_{\bullet})\bigr)
      \arrow[d] \arrow[r] &
      \bbDelta_{\cat{O}}\bigl(([k],L_{\bullet}),([n],P_{\bullet})\bigr)
      \arrow[d] \\
      \bbDelta_{\cat{O}}^{\mathrm{pre}}\bigl(([k],L_{\bullet}),([m],O_{\bullet})\bigr)
      \arrow[d] \arrow[r] \arrow[dr,phantom,very near
      start,"\lrcorner_{\text{ iff $f_{\bullet}$ invertible}}"] &
      \bbDelta_{\cat{O}}^{\mathrm{pre}}\bigl(([k],L_{\bullet}),([n],P_{\bullet})\bigr)
      \arrow[d] \\
      \bbDelta([k],[m]) \arrow[r] & \bbDelta([k],[n])\text{.}
    \end{tikzcd}
  \end{equation}

  Since $\bbDelta_{\cat{O}}^{\mathrm{pre}}\to\bbDelta$ is constructed
  --- and described --- as a Grothendieck construction, the lower
  square is cartesian if and only if $f_{\bullet}$ invertible. It is
  then enough to observe that the condition for a morphism in
  $\bbDelta^{\mathrm{pre}}_{\cat{O}}$ to be in $\bbDelta_{\cat{O}}$ is
  detected by postcomposition with a componentwise invertible morphism
  of $\bbDelta_{\cat{O}}$, which is clear.
  
  This characterisation now makes it easy to determine cartesian
  lifts. Let $\phi\colon[m]\to[n]$ be a morphism in $\bbDelta$, and
  let $([n],O_{\bullet})$ be a lift of $[n]$ in
  $\bbDelta_{\cat{O}}$. We define a cartesian lift of
  $([n],O_{\bullet})$ along $\phi$ to be $([m],O_{\phi(\bullet)})$
  with $(\phi,\id_{O_{\bullet}\circ\phi})$.
\end{proof}

\begin{corlr}
  \label{corlr:factor-syst-forest}
  Say that an arrow $(\phi,f_{\bullet})$ in $\bbDelta_{\cat{O}}$ is
  inert if $\phi$ is inert in $\bbDelta$ and active if $\phi$ is
  active in $\bbDelta$ and $f_{\bullet}$ is an equivalence. Then
  $((\op{\bbDelta_{\cat{O}}})^{\mathrm{inrt}},
  (\op{\bbDelta_{\cat{O}}})^{\mathrm{act}})$ defines a factorisation
  system on $\op{\bbDelta_{\cat{O}}}$.
\end{corlr}

\begin{proof}
  By~\cite[Proposition 2.1.2.5]{lurie17:_higher_algeb}, since
  $\func{d}_{\cat{O}}\colon\bbDelta_{\cat{O}}\subset
  \bbDelta^{\mathrm{pre}}_{\cat{O}}\to\bbDelta$ is a cartesian
  fibration, the factorisation system on its base can be lifted as
  required.
\end{proof}

\begin{definition}
  Let $\cat{O}$ be a combinatorial algebraic pattern. Its \textbf{plus
    construction} ${\op{\bbDelta_{\cat{O}}}}^{\natural}$ is the
  $(\infty,1)$-category $\op{\bbDelta_{\cat{O}}}$ equipped with the
  inert-active factorisation of~\cref{corlr:factor-syst-forest} and as
  elementary objects those $([n],O_{\bullet})$ with
  $[n]\in{\op{\bbDelta}}^{\natural,\mathrm{el}}$ (\emph{i.e.}  either
  $[0]$ or $[1]$) and $O_{n}\in\cat{O}^{\flat,\mathrm{el}}$.

  Its \textbf{misodendric plus construction} is the algebraic pattern
  induced on $\op{\bbDelta^{(1)}_{\cat{O}}}$.

  We shall say that an algebraic pattern is \textbf{well-directed} if
  it can be written as the plus construction of some combinatorial
  algebraic pattern.
\end{definition}

\begin{exmp}
  The main result of~\cite{chu18:_two} shows that
  ${\op{\bbDelta_{{\op{\bbGamma}}^{\natural}}}}^{\natural}$ is
  Morita-equivalent to ${\op{\bbOmega}}^{\natural}$ in the sense that
  their $(\infty,1)$-categories of Segal $\infty$-groupoids are
  equivalent. We will henceforth conflate operad objects with Segal
  ${\op{\bbDelta_{{\op{\bbGamma}}^{\natural}}}}^{\natural}$-objects.
\end{exmp}

\begin{scholium}
  \label{remark:segal-cond-plus-constr}
  The Segal condition for a Segal
  ${\op{\bbDelta_{\cat{O}}}}^{\natural}$-object $\func{X}$ can be
  explicitly written as the following set of conditions:
  \begin{description}
  \item[level decomposition]
    For any $([n],O_{\bullet})$, the canonical arrow
    \begin{equation}
      \label{eq:forest-segal-cond-lvl}
      \func{X}\bigl([n],O_{\bullet}\bigr)\to
      \func{X}
      \bigl([1],(O_{0}\rightsquigarrow
      O_{1})\bigr)\lmtimes_{\func{X}([0],(O_{1}))}\cdots
      \lmtimes_{\func{X}([0],(O_{n-1}))}
      \func{X}\bigl([1],(O_{n-1}\rightsquigarrow
      O_{n})\bigr) 
    \end{equation}
    is an equivalence.
  \item[root decomposition] For any height-$0$ forest of the form
    $([0],O_{0})$, the canonical map
    \begin{equation}
      \label{eq:forest-segal-cond-root}
      \func{X}\bigl([0],(O_{0})\bigr)\to
      \lim_{E\rightsquigarrow O_{0}}\func{X}\bigl([0],(E)\bigr)
    \end{equation}
    is an equivalence (where the limit is taken over
    $\cat{O}^{\mathrm{act}}_{\mathrm{el/O}}$)
  \item[shrub decomposition] For any
    $\bigl([1],(O_{0}\rightsquigarrow O_{1})\bigr)$, the canonical
    arrow
    \begin{equation}
      \label{eq:forest-segal-cond-shrub}
      \func{X}\bigl([1],(O_{0}\rightsquigarrow O_{1})\bigr)\to
      \lim_{E\rightsquigarrow O_{1}}\func{X}\bigl([1],(O_{0,E}\rightsquigarrow
      E)\bigr)
    \end{equation}
    is an equivalence, where $O_{0,E}$ is the fibre product
    $O_{0}\times^{\mathrm{act}}_{O_{1}}E$ --- with
    $\times^{\mathrm{act}}$ indicating that it is a fibre product in
    $\cat{O}^{\mathrm{act}}$.
  \end{description}
\end{scholium}

\begin{proof}
  The only part that bears commenting upon is the shrub
  decomposition. Recall that the inert maps in $\bbDelta_{\cat{O}}$
  are simply those whose projection to $\bbDelta$ is inert, so the
  relevant inert morphisms in $\bbDelta_{\cat{O}}$ are those from the
  elementaries over $[1]\in\bbDelta$. If we consider such a morphism,
  represented by
  \begin{equation}
    \label{eq:3}
    \begin{tikzcd}
      P \arrow[d,squiggly] \arrow[r,squiggly] & O_{0}
      \arrow[d,squiggly] \\
      E \arrow[r,squiggly] & O_{1}\text{,}
    \end{tikzcd}
  \end{equation}
  the equifibred condition implies that $P$ is the active pullback
  $O_{0}\times^{\mathrm{act}}_{O_{1}}E$.
\end{proof}

\begin{exmp}
  When $\cat{O}$ is ${\op{\bbGamma}}^{\natural}$, or more generally a
  momentous enrichable patter, the limits appearing in the root and
  shrub decompositions take on a simpler form: the root decomposition
  now becomes the condition that
  \begin{equation}
    \label{eq:gamma-forest-segal-cond-root}
    \func{X}\bigl([0],(O_{0})\bigr)\to
    \lim_{E\rightsquigarrow O_{0}}\func{X}\bigl([0],(E^{O_{0}}_{i})\bigr)
  \end{equation}
  be an equivalence (where $E^{O}_{i}\rightsquigarrow O$ is the unique
  lift of $O$ along the map $\lambda_{i}$
  of~\cref{exmp:gamma-act-list-retrac}, and we put a trivial factor if
  this lift is inexistent), and the shrub decomposition is the
  condition that
  \begin{equation}
    \label{eq:gamma-forest-segal-cond-shrub}
    \func{X}\bigl([1],(O_{0}\rightsquigarrow O_{1})\bigr)\to
    \prod_{i=1}^{\uly{O_{1}}}\func{X}\bigl([1],(O_{0,i}\rightsquigarrow
    E^{O_{1}}_{i})\bigr)
  \end{equation}
  be an equivalence, where $O_{0,i}$ is the fibre product
  $O_{0}\times^{\mathrm{act}}_{O_{1}}E^{O_{1}}_{i}$, as since
  $\cat{O}$ is momentous and enrichable, $E\rightsquigarrow O_{1}$
  must be one of the lifts $E^{O_{1}}_{i}$.
\end{exmp}

\begin{pros}[{\cite[Lemma 2.11]{chu18:_two}}]
  Forests and trees have the same Segal objects: the
  $(\infty,1)$-functor
  $\seg{{\op{\bbDelta_{\cat{O}}}}^{\natural}}(\infgrpds)
  \to\seg{{\op{\bbDelta^{(1)}_{\cat{O}}}}^{\natural}}(\infgrpds)$ of
  restriction along the inclusion
  $\bbDelta^{(1)}_{\cat{O}}\hookrightarrow\bbDelta_{\cat{O}}$ is an
  equivalence of $(\infty,1)$-categories.
\end{pros}

\begin{proof}
  Let $([n],O_{\bullet})$ be an $\cat{O}$-forest, and denote the
  decomposition of $\uly{O_{n}}$ into its fibres. We then write any
  forest as a union of trees, and use the forest decomposition
  condition of~\cref{remark:segal-cond-plus-constr}.
\end{proof}

\begin{construct}
  \label{construct:corollas-combin-str}
  We define a \textbf{corolla} of an $\cat{O}$-forest
  $F\in\bbDelta_{\cat{O}}$ to be an equivalence class of inert
  morphisms $F\to E$ where $E$ is elementary lying over
  $[1]\in\bbDelta$.
  
  We can now define a functor
  $\op{\uly{-}}\colon\bbDelta_{\cat{O}}\to\bbGamma$ by counting the
  numbers of corollas in an $\cat{O}$-forest, as in~\cite[Definition
  2.2.10]{chu20:_enric}.
\end{construct}

\begin{pros}
  \label{pros:enrich-plus-constr}
  The functor $\uly{-}\colon\op{\bbDelta_{\cat{O}}}\to\op{\bbGamma}$
  gives a structure of enrichable pattern on
  ${\op{\bbDelta_{\cat{O}}}}^{\natural}$.
\end{pros}

\begin{proof}
  By direct verification; this follows essentially by definition of
  corollas.
\end{proof}

\begin{pros}[{\cite[Theorem 10.16]{barwick18:_from}}]
  There is an equivalence of $(\infty,1)$-categories between
  ${\op{\bbDelta_{\cat{O}}}}^{\flat}$-Segal objects and weak Segal
  $\cat{O}$-fibrations.
\end{pros}

\begin{proof}[Proof (Sketch of construction)]
  The proof is a direct generalisation of that of~\cite[Theorem
  3.1.1]{hinich22:_luries}, which is quite technically involved, so we
  do not reproduce it here. However, since it will be useful later on,
  we recall the construction of the adjunction.

  We first explain what is, for our purposes, the most important part:
  how to get from Segal objects to weak fibrations.

  We start by constructing an $\infty$-functor
  $\omega^{\ast}\colon\prshvs{\bbDelta_{\cat{O}}}
  \to\prshvs{\bbDelta_{/\cat{O}}}$. Recall that presheaves on
  $\bbDelta_{\cat{O}}$ can be seen equivalently as presheaves on
  $\prshvs{\bbDelta_{\cat{O}}}$ satisfying the representability
  condition of being colimit-preserving. We now define an
  $\infty$-functor
  $\omega\colon\bbDelta_{/\cat{O}} \to\prshvs{\bbDelta_{\cat{O}}}$ as
  the composition $\func{i}^{\ast}\circ\yo_{\bbDelta_{/\cat{O}}}$ where
  $\func{i}^{\ast}$ is restriction (or inverse image) of presheaves
  along the canonical $\infty$-functor
  $\func{i}\colon\bbDelta_{\cat{O}}
  \hookrightarrow\bbDelta_{/\cat{O}^{\mathrm{act}}}
  \xrightarrow{(\cat{O}^{\mathrm{act}}\to\cat{O})_{!}}\bbDelta_{/\cat{O}}$.

  We thus get an inverse image $\infty$-functor
  $\omega^{\ast}\colon\prshvs{\prshvs{\bbDelta_{\cat{O}}}}
  \to\prshvs{\bbDelta_{/\cat{O}}}$, which we can restrict to the
  (full) sub-$(\infty,1)$-category of those presheaves which are
  representable and satisfy the Segal condition.  To finish, recall
  also (from~\cite[Corollary 5.1.6.12]{lurie09:_higher}) that
  $\prshvs{\bbDelta_{/\cat{O}}}\simeq{\prshvs{\bbDelta}}_{/\nerve\cat{O}}$,
  so we eventually have
  \begin{equation}
    \label{eq:1}
    \omega^{\ast}\colon\seg{{\op{\bbDelta_{\cat{O}}}}^{\natural}}(\infgrpds)
    \to{\prshvs{\bbDelta}}_{/\nerve\cat{O}}\text{.}
  \end{equation}

  In the case where $\cat{O}={\op{\bbGamma}}^{\natural}$, our functor
  $\omega$ corresponds to the one described
  in~\cite[\S{}5.1]{heuts16:_luries}. By~\cite[\S{}3.3]{hinich22:_luries},
  this functor does land first in the subcategory
  $\seg{{\op{\bbDelta}}^{\natural}}(\infgrpds)_{/\nerve\cat{O}}$ which
  we interpret as $\infcats_{/\cat{O}}$, and in fact in the further
  subcategory of those $\infty$-functors to $\cat{O}$ which are weak
  Segal fibrations.
  
  To go the other way, we consider
  $\varpi\colon\bbDelta_{\cat{O}}\xrightarrow{\yo}
  \prshvs{\bbDelta_{\cat{O}}} \to\wkfibs{\cat{O}}$ and the
  ``conjoint'' profunctor it corepresents: then $(\omega^{\ast})^{-1}$
  is the $\infty$-functor mapping $(\cat{E}\to\cat{O})$ to the
  presheaf
  $\wkfibs{\cat{O}}(\varpi(-),\cat{E})
  \colon\bbDelta_{\cat{O}}\to\infgrpds$.
\end{proof}

Note that --- again for $\cat{O}={\op{\bbGamma}}^{\natural}$ --- an
alternate proof, using the monoidal envelope $\infty$-functor, appears
in\cite[Corollary 4.2.8]{haugseng21}. However, this proof uses
crucially an explicit description of the Grothendieck construction of
$\func{Env}(\mathbb{1})$ seen as a Segal $\cat{O}$-$\infty$-category,
which do not have in general.

\section{Monoidal envelopes and Grothendieck opfibrations}
\label{sec:mono-envel-groth}

\subsection{Construction of the monoidal envelope functor}
\label{sec:constr-mono-envel}

As explained in the introduction, we wish to interpret the usual
formula computing monoidal envelopes of operads as a fibrewise oplax
extension along a certain functor appropriately mixing in
partitions. To define the necessary fibrewise extensions, we need to
briefly work in the generality of formal $\infty$-category theory
introduced in~\cite{riehl22:_elemen}, that is in the framework of
$\infty$-cosmoi, modelling the $(\infty,2)$-category of
$\infty$-categories.  Recall from~\cite[Proposition 9.1.8, Theorem
9.3.3]{riehl22:_elemen} that (pointwise) oplax extensions along an
$\infty$-functor can be expressed as colimits weighted by the conjoint
of this $\infty$-functor. We will define fibrewise extensions
similarly, replacing this conjoint by a relative version.

The reader who wishes to limit their intake of abstract nonsense may
use the formula of~\cref{eq:fibre-oplax-extn-formula} as a definition,
and take its functoriality properties on faith.

\begin{construct}[Relative comma $\infty$-category]
  In an $\infty$-cosmos $\cat{K}$, consider a cocorrespondence in the
  sliced $\infty$-cosmos $\cat{K}_{/B}$:
  \begin{equation}
    \label{eq:span-isofibr-slice-cosmos}
    \begin{tikzcd}
      & & \cat{E} \arrow[dd,two heads] \arrow[from=dll,"\func{f}"]
      \arrow[from=drr,"\func{g}"'] & & \\
      \cat{F} \arrow[drr,bend left,two heads,"\func{p}"'] & & & &
      \cat{G}
      \arrow[dll,bend right,two heads,"\func{q}"]  \\
      & & \cat{B} & &
    \end{tikzcd}
  \end{equation}

  We let $\func{f}\downarrow_{/\cat{B}}\func{g}$ denote the comma
  object in $\cat{K}_{/B}$, and call it the \textbf{relative comma
    $\infty$-category} over $B$.
\end{construct}

\begin{remark}
  By~\cite[Proposition 1.2.22, (iv)--(vi)]{riehl22:_elemen}, the
  relative comma $\infty$-category can be constructed as
  \begin{equation}
    \label{eq:rel-comma-formula-construct}
    \func{f}\downarrow_{/\cat{B}}\func{g}
    \simeq(\cat{F}\times_{\cat{B}}\cat{G})\lmtimes_{(\cat{E}\times_{\cat{B}}\cat{E})}
    (\cat{B}\times_{\funcs{\mathbb{2}}{B}}\funcs{\mathbb{2}}{E})
  \end{equation}
  (where the map $\cat{B}\to\funcs{\mathbb{2}}{B}$ is the
  diagonal). That is, informally, an object of
  $\func{f}\downarrow_{/\cat{B}}\func{g}$ consists of a triple
  $(F,G,\alpha)$ where $F$ and $G$ are objects of $\cat{F}$ and
  $\cat{G}$ respectively and $\alpha\colon \func{f}(F)\to\func{g}(G)$
  is an arrow in $\cat{E}$, such that all data live above the same
  object of $\cat{B}$ (\emph{i.e.} there is an object $B\in\cat{B}$
  and isomorphisms $\func{p}F\xrightarrow{\simeq}B$,
  $\func{q}G\xrightarrow{\simeq}B$, and
  $\alpha\xRightarrow{\simeq}\id_{B}$).
\end{remark}

\begin{lemma}
  The canonical projection
  $\func{f}\downarrow_{/\cat{B}}\func{g}\to\cat{F}\times\cat{G}$ is a
  discrete two-sided fibration in $\cat{K}$.
\end{lemma}

\begin{proof}
  As in~\cite[Proposition 7.4.6]{riehl22:_elemen}.
\end{proof}

\begin{definition}[Fibrewise (op)lax extensions]
  Let $\cat{E}\xrightarrow{\func{p}}\cat{B}$ and
  $\cat{F}\xrightarrow{\func{q}}\cat{B}$ be two $\infty$-categories
  defined over a base $\cat{B}$, and let
  $\func{K}\colon\cat{E}\to\cat{F}$ be an $\infty$-functor defined
  over $\cat{B}$. Let $\func{D}\colon\cat{E}\to\cat{G}$ be an
  $(\infty,1)$-functor, so that we have the solid diagram
  \begin{equation}
    \label{eq:fibrwise-lax-extn-setup}
    \begin{tikzcd}[column sep=small]
      \cat{E} \arrow[ddr,bend right=5,"{\func{p}}"']
      \arrow[drr,"{\func{K}}"'] \arrow[rrrr,"{\func{D}}"] & & &
      & \cat{G} \\
      & & \cat{F} \arrow[dl,bend left,"{\func{q}}"] \arrow[urr,dashed]
      & & \\
      & \cat{B} & & &
    \end{tikzcd}
  \end{equation}

  A \textbf{fibrewise lax extension} of $\func{D}$ along $\func{K}$
  (relative to $\cat{B}$) is a limit
  \begin{equation}
    \label{eq:fibrwise-lax-extn-def}
    \laxext{\func{K}/\cat{B}}\func{D}
    \coloneqq\wtdlim{\func{K}^{/\cat{B}}_{\ast}}{\func{D}}
    \colon\cat{F}\to\cat{G}
  \end{equation}
  of $\func{D}$ weighted by the $\infty$-profunctor
  $\func{K}^{/\cat{B}}_{\ast}\colon\cat{E}\slashedrightarrow\cat{F}$
  associated with the relative comma $\infty$-category
  $\id_{\cat{F}}\downarrow_{/\cat{B}}\func{K}$.
  
  A \textbf{fibrewise oplax extension} of $\func{D}$ along $\func{K}$
  (relative to $\cat{B}$) is a colimit
  \begin{equation}
    \label{eq:fibrwise-oplax-extn-def}
    \oplaxext{\func{K}/\cat{B}}\func{D}
    \coloneqq\wtdcolim{\func{K}_{/\cat{B}}^{\ast}}{\func{D}}
    \colon\cat{F}\to\cat{G}
  \end{equation}
  of $\func{D}$ weighted by the $\infty$-profunctor
  $\func{K}_{/\cat{B}}^{\ast}\colon\cat{F}\slashedrightarrow\cat{E}$
  corresponding to the relative comma $\infty$-category
  $\func{K}\downarrow_{/\cat{B}}\id_{\cat{F}}$.
\end{definition}

\begin{remark}
  We have the explicit formulae, deduced from~\cite[Lemma
  9.5.5]{riehl22:_elemen}, computing fibrewise extensions: the
  fibrewise oplax extension of $\func{D}$ along $\func{K}$, evaluated
  at $F$, is
  \begin{equation}
    \label{eq:fibre-oplax-extn-formula}
    \oplaxext{\func{K}/\cat{B}}\func{D}(F)
    =\colim_{\substack{\func{K}(E)\to
        F\\E\in\cat{E}_{\func{q}(F)}}}\func{D}(E)\text{,} 
  \end{equation}
  and the fibrewise lax extension of $\func{D}$ along $\func{K}$,
  evaluated at $F$, is
    \begin{equation}
    \label{eq:fibre-lax-extn-formula}
    \laxext{\func{K}/\cat{B}}\func{D}(F)
    =\lim_{\substack{F\to\func{K}(E)\\E\in\cat{E}_{\func{q}(F)}}}\func{D}(E)\text{.}
  \end{equation}
\end{remark}

\begin{definition}[Envelope]
  Let $\func{X}\colon\op{\bbDelta_{\cat{O}}}\to\infgrpds$ be a
  precosheaf on $\op{\bbDelta_{\cat{O}}}$. Its \textbf{envelope} is
  the precosheaf
  \begin{equation}
    \label{eq:def-repr-mondl-envlp}
    \func{Env}(\func{X})
    =\func{i}^{\ast}\oplaxext{\func{i}/\op{\bbDelta}}\func{X}
  \end{equation}
  where $\func{i}^{\ast}$ denotes the $(\infty,1)$-functor of
  (fibrewise) restriction along
  $\func{i}\colon\op{\bbDelta_{\cat{O}}}
  \hookrightarrow\op{(\bbDelta_{\cat{O}}^{\mathrm{pre}})}$,
  right-adjoint to fibrewise oplax extension.
\end{definition}

\begin{remark}
  From the formula for fibrewise extensions
  (~\cref{eq:fibre-oplax-extn-formula}), we see that the value taken
  by the envelope of
  $\func{X}\in\seg{{\op{\bbDelta_{\cat{O}}}}^{\natural}}(\cat{C})$ at
  an $\cat{O}$-tree $T_{\bullet}$ of length $[n]$ is computed by
  \begin{equation}
    \label{eq:envelope-explicit-formula}
    \func{Env}(\func{X})(T_{\bullet})
    =\colim_{\substack{B_{\bullet}\rightsquigarrow
        T_{\bullet}
        \\\in((\bbDelta_{\cat{O}}^{\mathrm{pre}})_{[n]})_{/T_{\bullet}}}}
    \lim_{E\rightsquigarrow B_{n}}\func{X}(B_{\bullet,E})
  \end{equation}
  where $(B_{\bullet,E})_{E\rightsquigarrow B_{n}}$ denotes the forest
  decomposition in fibres (as in~\cref{eq:forest-segal-cond-shrub}),
  and where we denote the indexing arrows for the colimit as
  $\rightsquigarrow$ to emphasise that we view them as families of $n$
  active arrows of $\cat{O}$.
\end{remark}

\begin{remark}
  In the case $\cat{O}={\op{\bbGamma}}^{\flat}$, we may understand the
  envelope in the following way. Recall that an active morphism in
  $\op{\bbGamma}$ can be seen as giving a partition of its source
  indexed by its target (possibly with empty parts). Then the colimit
  creates several copies of $\func{X}(\corol{n})$, each equipped with
  a new partition specifying how to distribute its inputs.
\end{remark}

To work with envelopes in an intuitive way, it will then be convenient
to introduce the following terminology.

\begin{definition}[$\cat{O}$-Partitions]
  Let $\cat{O}$ be an algebraic pattern, and let $O$ and $P$ be
  objects of $\cat{O}$. An $\cat{O}$-partition of $O$ by $P$, or
  simply a \textbf{$P$-partition of $O$} when there is no ambiguity,
  is an active morphism $O\rightsquigarrow P$.
\end{definition}

\begin{exmp}
  \label{exmp:envelope-of-terminal}
  If $\mathbb{1}$ denotes the terminal
  ${\op{\bbDelta_{\cat{O}}}}^{\natural}$-Segal $\infty$-groupoid, then
  its envelope $\func{Env}(\mathbb{1})$ evaluates, on any forest
  $T_{\bullet}$ of length $n$, the the groupoidal localisation of
  $(\bbDelta_{\cat{O}}^{\mathrm{pre}})_{[n]/T_{\bullet}}$.
\end{exmp}

\begin{proof}
  By definition, $\func{Env}(\mathbb{1})(T_{\bullet})$ is given by the
  colimit
  \begin{equation}
    \label{eq:envlp-of-terminal}
    \func{Env}(\mathbb{1})(T_{\bullet})
    =\colim_{B_{\bullet}\rightsquigarrow
      T_{\bullet}}\mathbb{1}
  \end{equation}
  (where the last $\mathbb{1}$ is the terminal $\infty$-groupoid).

  Now $\func{Env}(\mathbb{1})(T_{\bullet})$ is an $\infty$-groupoid so
  it is equivalent to its localisation, and furthermore the
  localisation functor $\func{L}\colon\infcats\to\infgrpds$, being
  left-adjoint to the inclusion
  $\func{I}\colon\infgrpds\hookrightarrow\infcats$, preserves
  colimits, so we may compute the colimit on the right-hand side in
  $\infcats$ and then localise.
  
  As a colimit of a constant diagram of $(\infty,1)$-categories, the
  colimit in $\infcats$ is equivalent to the tensor of the constant
  value $\mathbb{1}$ by the indexing $(\infty,1)$-category
  $(\bbDelta_{\cat{O}}^{\mathrm{pre}})_{[n]/T_{\bullet}}$.

  From the adjunction $\func{L}\dashv\func{I}$ and the tensor-cotensor
  adjunction, we have that for any $(\infty,1)$-category $\cat{C}$ and
  any $\infty$-groupoids $\cat{G}$ and $\cat{H}$,
  \begin{equation}
    \label{eq:8}
    \infgrpds\bigl((\func{L}\cat{C})\otimes\cat{G},\cat{H}\bigr)
    \simeq\infgrpds\bigl(\func{L}\cat{C},\funcs{G}{H}\bigr)
    \simeq\infcats\bigl(\cat{C},\func{I}(\funcs{G}{H})\bigr)\text{.}
  \end{equation}
  Since functor $(\infty,1)$-categories between $\infty$-groupoids are
  automatically $\infty$-groupoids, $\func{I}(\funcs{G}{H})$ is also
  the cotensor of $(\infty,1)$-categories
  $\funcs{\func{I}G}{(\func{I}H)}$, and thus
  \begin{equation}
    \label{eq:9}
    \begin{split}
      \infcats\bigl(\cat{C},\func{I}(\funcs{G}{H})\bigr)
      &\simeq\infcats\bigl(\cat{C},\funcs{\func{I}G}{(\func{I}H)}\bigr)
      \simeq\infcats\bigl(\cat{C}\otimes\func{I}\cat{G},
      \func{I}\cat{H}\bigr)\\
      &\simeq\infgrpds\bigl(\func{L}(\cat{C}\otimes\func{I}\cat{G}),
      \cat{H}\bigr)\text{,}
    \end{split}
  \end{equation}
  which shows that
  $\func{L}(\cat{C}\otimes\func{I}\cat{G})\simeq
  (\func{L}\cat{C})\otimes\cat{G}$.

  Hence, in our case, we obtain --- now writing
  $\cat{C}[^{-1}]=\func{L}\cat{C}$ for the total localisation of an
  $(\infty,1)$-category $\cat{C}$ --- the tensor of $\infty$-groupoids
  $(\bbDelta_{\cat{O}}^{\mathrm{pre}})_{[n]/T_{\bullet}}[^{-1}]
  \otimes\mathbb{1}$, which is simply
  $(\bbDelta_{\cat{O}}^{\mathrm{pre}})_{[n]/T_{\bullet}}[^{-1}]$.
\end{proof}

\begin{definition}[Linearisable pattern]
  A collection $\{C_{i}\}_{i\in I}$ of objects in an
  $(\infty,1)$-category $\cat{C}$ is \textbf{mock multiterminal} if
  the projection $\cat{C}_{/\{C_{i}\}_{i}}\to\cat{C}$ is an
  equivalence.
  
  A combinatorial algebraic pattern
  $\cat{O}\xrightarrow{\uly{-}}{\op{\bbGamma}}^{\natural}$ is
  linearisable if the collection of $\flat$-elementary objects is mock
  multiterminal in $\cat{O}^{\text{act}}$.
\end{definition}

We note that while, obviously, if a collection $\{C_{i}\}_{i\in I}$ is
multiterminal, it is in particular mock multiterminal, this is is not
the only case: it can also happen if the only morphisms between
objects of $\{C_{i}\}_{i}$ are ismorphisms.

\begin{notation}
  When $\cat{O}$ is linearisable, we have a simple classification of
  the elementaries of ${\op{\bbDelta_{\cat{O}}}}^{\natural}$. Recall
  that those were defined to be the trees (\emph{i.e.} the forests
  whose final value is elementary in $\cat{O}^{\flat}$) of length
  $[0]$ and $[1]$.

  If $E_{\bullet}$ is of length $[0]$, it is simply given by an
  elementary object of $\cat{O}^{\flat}$, so we write those as
  $\edge_{E}$ with $E\in\cat{O}^{\flat,\mathrm{el}}$, and call them
  the \textbf{edges}.

  If $E_{\bullet}$ is of length $[1]$, it must be given by a diagram
  in $\cat{O}^{\mathrm{act}}$ of the form $O\rightsquigarrow E$ with
  $E\in\cat{O}^{\flat,\mathrm{el}}$. The linearisability condition
  ensures that for any given $O$, there is only essentially one such
  morphism, and so we denote the corresponding object as $\corol{O}$,
  and dub such elementaries the \textbf{corollas} (as they are the
  elementaries considered in~\cref{construct:corollas-combin-str}).
\end{notation}

\begin{pros}
  For any combinatorial pattern $\cat{O}$, the misodendric plus
  construction ${\op{\bbDelta_{\cat{O}}^{(1)}}}^{\natural}$ is
  linearisable.
\end{pros}

\begin{proof}
  Let $([n],O_{\bullet})$ be an object of $\bbDelta_{\cat{O}}$. By the
  condition for misodendric objects, $O_{n}$ must be in
  $\cat{O}^{\flat, \mathrm{el}}$. Recall
  from~\cref{remark:morph-plus-constr-minl-data} that a morphism to
  $([n],O_{\bullet})$ is determined by the data of a morphism to $[n]$
  in $\bbDelta$ and a (semi-inert and active) morphism to $O_{n}$.

  Now the elementaries of ${\op{\bbDelta_{\cat{O}}}}^{\flat}$ are of
  the form $\left([1],(E_{0}\rightsquigarrow E_{1})\right)$ with
  $E_{1}\in\cat{O}^{\flat, \mathrm{el}}$. There is only one active
  morphism $[1]\rightsquigarrow[n]$, which is specified by
  $(0<1)\mapsto(0<n)$. Finally, an active morphism in
  ${\op{\bbDelta_{\cat{O}}}}^{\flat}$ must have its morphism to
  $O_{n}$ be an isomorphism, so $E_{1}$ is also forced to be $O_{n}$
  (while $E_{0}$ was already seen to be equivalent to $O_{0}$).
\end{proof}

In particular, when $\cat{O}$ is linearisable,
evaluating~\cref{exmp:envelope-of-terminal} on the elementaries, we
find that:
\begin{itemize}
\item
  $\coprod_{E\in\cat{O}^{\flat,\mathrm{el}}}\func{Env}(\mathbb{1})(\edge_{E})$
  is the localisation of $\cat{O}^{\mathrm{act}}$;
\item for any corolla
  $C_{O}=\bigl([1],(O\rightsquigarrow E_{O})\bigr)$ (with $E_{O}$ the
  unique elementary receiving an active morphism from $O$), an object
  of $\func{Env}(\mathbb{1})(C_{O})$ consists of an active morphism
  $P_{0}\rightsquigarrow P_{1}$ of $\cat{O}$ --- a $P_{1}$-partition
  of $P_{0}$ --- along with an $O$-partition $P_{0}\rightsquigarrow O$
  (such that, necessarily, $E_{P_{1}}$ is also compatibly equivalent
  to $E_{O}$).
\end{itemize}

\subsection{Opfibrations and representable monoidality}
\label{sec:opfibr-repr-mono}

\begin{pros}
  For any Segal ${\op{\bbDelta_{\cat{O}}}}^{\natural}$-object
  $\func{X}$, its envelope $\func{Env}(\func{X})$ is a Segal
  ${\op{\bbDelta_{\cat{O}}}}^{\natural}$-object.
\end{pros}

\begin{proof}
  Let $T_{\bullet}$ be an object of $\bbDelta_{\cat{O}}$ of length
  $[n]$. The Segal condition requires comparing
  \begin{equation}
    \label{eq:enlvop-segal-test}
    \func{Env}(\func{X})(T_{\bullet})
    =\colim_{B_{\bullet}\to T_{\bullet}}\func{X}(B_{n})
    =\colim_{B_{\bullet}\to T_{\bullet }}\lim_{B_{\bullet}\to E}\func{X}(E_{\bullet})
  \end{equation}
  and
  \begin{equation}
    \label{eq:enlvop-segal-test-2}
    \lim_{T_{\bullet}\to E_{\bullet}}\func{Env}(\func{X})(E_{\bullet})
    =\lim_{T_{\bullet}\to E_{\bullet}}\colim_{B_{0,1}\to E_{\bullet}}\func{X}(B_{0,1})\text{.}
  \end{equation}
  Their equality is the condition of distributivity of limits over
  colimits as made explicit in~\cite[Definition
  7.12]{chu21:_homot_segal}, and using~\cite[Corollary
  7.17]{chu21:_homot_segal} which shows that $\infgrpds$ is admissible
  (or even any $(\infty,1)$-topos when the $\infty$-category of active
  maps from a given object of $\cat{O}$ to an elementary is finite).
  
  Alternatively, we can use the explicit form of the Segal conditions
  given in~\cref{remark:segal-cond-plus-constr}, which can be checked
  directly.
\end{proof}

For the cocartesian properties of the envelope, we now need to
specialise to the case where $\cat{O}$ is linearisable. Indeed,
following~\cite[Theorem 2.4]{hermida04:_fibrat}, we will characterise
cocartesian fibrations of Segal $\cat{P}$-objects (for
$\cat{P}={\op{\bbDelta_{\cat{O}}}}^{\natural}$ a well-directed
algebraic pattern) in terms of cocartesian fibrations of the
underlying $\cat{P}_{\cat{P}_{0}/}$-objects of their envelopes. This
presupposes having already a good understanding of cocartesian
fibrations of $\cat{P}_{\cat{P}_{0}/}$-objects, which is only the case
when $\cat{P}_{\cat{P}_{0}/}$ is ${\op{\bbDelta}}^{\natural}$, in
particular when $\cat{P}={\op{\bbOmega}}^{\natural}$ (and
$\cat{P}_{0}=\{\edge\}$).

\begin{remark}
  In a combinatorial pattern $\cat{O}$, the $\flat$-elementaries lie
  over $\langle1\rangle\in\op{\bbGamma}$. An active and semi-inert
  morphism to $\langle1\rangle$ can have as source only
  $\langle1\rangle$ or $\langle0\rangle$. So if $\cat{O}$ is
  linearisable, for any object $O$ the unique active morphism to an
  elementary will only be semi-inert in $\uly{O}$ is $\langle1\rangle$
  or $\langle0\rangle$, in which case it \emph{will} be so because
  there are no non-semi-inert morphisms
  $\langle0\rangle\to\langle1\rangle$ or
  $\langle1\rangle\to\langle1\rangle$.
\end{remark}

\begin{lemma}
  Suppose $\cat{O}$ is linearisable and $\cat{O}_{\langle1\rangle}$ is
  an $\infty$-groupoid. The slice $(\bbDelta_{\cat{O}})_{/\edge}$ is
  equivalent to $\bbDelta\times\cat{O}_{\langle1\rangle}$, itself
  equivalent to $\colim_{\cat{O}_{\langle1\rangle}}\bbDelta$.
\end{lemma}

\begin{proof}
  Consider a morphism $T\to\edge$ from some tree
  $T=([n],O_{\bullet})$, consisting of a map $\phi\colon[n]\to[0]$ and
  a transformation $O_{\bullet}\to\phi^{\ast}E$ where $E$ is an
  elementary defining the edge $\edge_{E}=([0],(E))$. Now
  $\phi^{\ast}E$ is a constant diagram, so the equifibred condition
  for the transformation implies that all components of $O_{\bullet}$
  must be equal to $O_{n}$. By the previous remark, $O_{n}$ only
  admits a semi-inert active morphism to an $E$ if it lies over
  $\langle1\rangle$ (and it exists uniquely by the linearisability
  condition).
\end{proof}

By~\cite[Proposition 2.37]{barkan22:_arity_approx_operad}, there is a
canonical pattern structure on a coslice of a pattern. The equivalence
of categories underlies one of algebraic patterns.

\begin{remark}
  Since hom $\infty$-functors preserve limits, we have that
  $\seg{\colim_{\cat{O}_{\langle1\rangle}}{\op{\bbDelta}}^{\natural}}(\cat{C})
  \simeq\lim_{\cat{O}_{\langle1\rangle}}\seg{{\op{\bbDelta}}^{\natural}}(\cat{C})$. In
  other words, a Segal $\op{(\bbDelta_{\cat{O}})}_{\edge/}$-object can
  be seen as a family of category objects indexed by the
  $\infty$-groupoid $\cat{O}_{\langle1\rangle}$, which we shorten to
  \textbf{$\cat{O}_{\langle1\rangle}$-category}.
\end{remark}

\begin{definition}[Underlying $\cat{O}_{\langle1\rangle}$-category]
  Suppose $\cat{O}$ is linearisable and $\cat{O}_{\langle1\rangle}$ is
  an $\infty$-groupoid, and denote
  $\func{p}\colon\op{(\bbDelta_{\cat{O}})}_{\edge/}\to\op{\bbDelta_{\cat{O}}}$.
  The \textbf{underlying $\cat{O}_{\langle1\rangle}$-category}, or
  simply underlying category, object of a Segal
  ${\op{\bbDelta_{\cat{O}}}}^{\natural}$-object $\func{X}$ is its
  restriction along $\func{p}$.
\end{definition}

If $\func{F}$ is the name of a Segal
${\op{\bbDelta_{\cat{O}}}}^{\natural}$-object, we will generally
denote its underlying category object $\cat{F}$.

\begin{remark}
  The functor $\func{p}$ clearly admits unique liftings of inert
  morphisms, and it is extendable, so by~\cite[Proposition
  7.13]{chu21:_homot_segal} the restriction $\infty$-functor
  $\func{p}^{\ast}$ has a left-adjoint (given by oplax extension),
  which is fully-faithful, and has as its essential images the
  presheaves with empty value on any non-linear tree.
\end{remark}

\begin{definition}[Cocartesian fibration]
  A morphism $\func{X}\to\func{B}$ of Segal
  ${\op{\bbDelta_{{\op{\bbGamma}}^{\flat}}}}^{\natural}$-$\infty$-groupoids
  is a \textbf{cocartesian fibration of Segal objects} if
  $\cat{Env}(\func{X})\to\cat{Env}(\func{B})$ is a cocartesian
  fibration (of $(\infty,1)$-categories), where $\cat{Env}(\func{X})$
  denotes the underlying $(\infty,1)$-category of
  $\func{Env}(\func{X})$.

  A Segal object $\func{X}$ is \textbf{representably monoidal} if the
  unique morphism $\func{X}\to\mathbb{1}$ is a cocartesian fibration of
  Segal objects.
\end{definition}

\begin{pros}
  \label{pros:envlp-is-monoidal}
  For any Segal object $\func{X}$, $\func{Env}(\func{X})$ is
  representably monoidal.
\end{pros}

\begin{proof}
  Since families of cocartesian fibrations can be checked
  componentwise, we shall assume that $\cat{O}_{\langle1\rangle}$ is
  comprised of one single $\flat$-elementary object.
  
  Let $\phi\colon O\rightsquigarrow O^{\prime}$ be a morphism in
  $\cat{Env}(\mathbb{1})$, seen as an $\cat{O}$-partition of $O$ by
  $O^{\prime}$, and let $C\in\cat{Env}(\func{Env}(\func{X}))$ lying
  over $O$. Under the identification
  \begin{equation}
    \label{eq:6}
    \begin{split}
      \func{Env}\bigl(\func{Env}(\func{X})\bigr)(\edge)
      &\simeq\colim_{P\in\cat{O}^{\mathrm{act}}}
      \func{Env}(X)([0],P)\\
      &\simeq\colim_{P\in\cat{O}^{\mathrm{act}}}
      \colim_{Q\rightsquigarrow P} \lim_{E\rightsquigarrow
        Q}\func{X}(\edge)\text{,}
    \end{split}
  \end{equation}
  where the condition of ``lying over $O$'' means that the index $P$
  in the first colimit is fixed to be $O$, one passes to $\phi_{!}C$
  over $O^{\prime}$ by composing the maps $Q\rightsquigarrow O$ in the
  second colimit with $O\rightsquigarrow O^{\prime}$.

  For clarity, let us write it in the momentous enrichable case, where
  this simplifies to
  \begin{equation}
    \label{eq:5}
    \begin{split}
      \func{Env}\bigl(\func{Env}(\func{X})\bigr)(\edge)
      &\simeq\colim_{P\in\cat{O}^{\mathrm{act}}}
      \prod_{i=1}^{\uly{P}}\func{Env}(X)(\edge)\\
      &\simeq\colim_{P\in\cat{O}^{\mathrm{act}}}
      \prod_{i=1}^{\uly{P}}\colim_{Q_{i}\in\cat{O}^{\mathrm{act}}}
      \prod_{j=1}^{\uly{Q_{i}}}\func{X}(\edge)\text{,}
    \end{split}
  \end{equation}
  and so we write $C$ as
  \begin{equation}
    \label{eq:envlp-is-mondl-colr-dbl-envlp}
    (C_{1,1}\cdots C_{1,\uly{Q_{1}}})\cdots(C_{\uly{O},1}\cdots
    C_{\uly{O},\uly{Q_{\uly{O_{0}}}}})
  \end{equation}
  for some choice of $(Q_{i})_{i=1}^{\uly{O}}$.

  We set
  \begin{equation}
    \label{eq:envlp-is-mondl-colr-lift}
    \phi_{!}C=(C_{1,1}\cdots
    C_{n_{1},i_{n_{1}}})\cdots(C_{n_{m-1}+1,1}\cdots C_{n_{m},i_{n_{m}}})
  \end{equation}
  which lies over $\langle m\rangle=\uly{O^{\prime}}$, where
  $(n_{1},\dots,n_{m})$ is the partition of $n$ in $m$ furnished by
  $\uly{O\rightsquigarrow O^{\prime}}$.

  A morphism $C\to\phi_{!}C$ in $\cat{Env}(\func{Env}(\func{X}))$ is
  given by a moprhism
  $C_{1,1}\cdots C_{n_{m},i_{n_{m}}}\to C_{1,1}\cdots
  C_{n_{m},i_{n_{m}}}$ in $\cat{Env}(\func{X})$ along with a partition
  of $\sum_{k=1}^{m}i_{n_{k}}$ into $n$ parts. We define the lift
  $C\to\phi_{!}C$ of $\phi$ to be given by the identity arrow of
  $C_{1,1}\cdots C_{n_{m},i_{n_{m}}}$ along with the partition
  exhibited in~\cref{eq:envlp-is-mondl-colr-dbl-envlp}. This lift is
  cocartesian.
\end{proof}

Thus the construction $\func{X}\mapsto\func{Env}(\func{X})$ defines an
$(\infty,1)$-functor
$\func{Env}\colon\seg{{\op{\bbDelta_{\cat{O}}}}^{\natural}}(\infgrpds)
\to\cat{O}_{\langle1\rangle}\textnormal{-}\infmoncats$.

\begin{remark}[Monoidal structure on monoidal $\infty$-categories]
  \label{remark:product-in-mndl-oprd-cat}
  Using the idea from the proof of~\cref{pros:envlp-is-monoidal}, we
  can construct a product on the colours of a representably monoidal
  $\infty$-operad $\func{X}$. First note that (regardless of
  monoidality) the colours of $\func{Env}(\func{X})$ in the image of
  the unit map $\func{X}\to\func{Env}(\func{X})$ are exactly those
  whose image under the morphism
  $\func{Env}(\func{X})\to\func{Env}(\mathbb{1})$ is $1$.

  Consider an $n$-uple of colours of $\func{X}$, given by $n$
  morphisms $C_{1},\cdots,C_{n}\colon\yo\edge\to\func{X}$. Those define
  a morphism
  $(C_{1},\cdots,C_{n})\colon\yo\edge\to\func{Env}(\func{X})$, whose
  image lies over the colour $n$ of $\func{Env}(\mathbb{1})$. Now since
  $\func{X}$ is representably monoidal, the morphism $n\to1$ in
  $\func{Env}(\mathbb{1})$ has a cocartesian lift from
  $(C_{1},\cdots,C_{n})$, whose target is then a colour
  $C_{1}\otimes\cdots\otimes C_{n}$ of $\func{X}$. Clearly, the same
  construction can be applied to obtain a product of morphisms as
  well, with appropriate functoriality.
\end{remark}

\begin{thm}
  \label{thm:envlp-is-left-adj-free}
  The $(\infty,1)$-functor $\func{Env}$ is left-adjoint to the
  inclusion $\infmoncats\hookrightarrow\infoprds$.
\end{thm}

\begin{proof}
  By the construction of the envelope, we have a unit morphism
  $\eta_{\func{X}}\colon\func{X}\to\func{Env}(\func{X})$ for any
  $\infty$-operad $\func{X}$. We need to construct a counit
  $\varepsilon_{\func{V}}\colon\func{Env}(\func{V})\to\func{V}$ for
  any representably monoidal $\infty$-operad $\func{V}$, which is a
  morphism of monoidal $\infty$-categories.

  This morphism is provided by the construction
  of~\cref{remark:product-in-mndl-oprd-cat}: since $\func{V}$ is
  representably monoidal, it admits a monoidal product, and its
  envelope simply corresponds to adding a second level a
  parenthesising to the products.

  We will construct $\varepsilon_{\func{V}}$ componentwise, as a
  natural transformation of $\infty$-functors
  $\op{\bbDelta_{{\op{\Gamma}}^{\flat}}}\to\infgrpds$. Let
  $T_{\bullet}$ be a tree. By the
  formula~\cref{eq:envelope-explicit-formula}, giving a map
  $\func{Env}(\func{V})(T_{\bullet})\to\func{V}(T_{\bullet})$ is
  equivalent to giving, for each morphism
  $B_{\bullet}\rightsquigarrow T_{\bullet}$, a map
  $\func{X}(B_{\bullet})\to\func{X}(T_{\bullet})$. But recall that
  $B_{\bullet}\rightsquigarrow T_{\bullet}$ can be interpreted as a
  $T_{\bullet}$-partition of $B_{\bullet}$. Following the previous
  remark, we can use the product to simply reorganise the
  parenthesising levels according to the partition, which produces the
  desired morphism.

  Finally, it is directly checked that $\varepsilon$ and $\eta$
  satisfy the triangular equalities, so that they do exhibit an
  adjunction.
\end{proof}

\section{Cartesian monoidal structures and application to the
  Grothendieck construction}
\label{sec:cart-mono-struct-appl}

\subsection{Cartesian monoidal structures}
\label{sec:cart-mono-struct}

Defining the monoidal $(\infty,1)$-category associated with an
$(\infty,1)$-category admitting finite coproducts is easy enough:
by~\cite[Proposition 2.4.3.9]{lurie17:_higher_algeb}, the
$\infty$-functor realising this construction is right-adjoint to the
one taking the underlying $(\infty,1)$-category of a \emph{unital}
$\infty$-operad. For cartesian monoidal $\infty$-categories, however,
it is not so easy.

We follow here the idea used independently
in~\cite{barwick20:_spect_mackey_k_ii}
and~\cite{dyckerhoff19:_higher_segal_spaces} to define cartesian
monoidal $\infty$-categories: defining them not as $\infty$-operads,
but as ``$\infty$-antioperads'', the variant of $\infty$-operads whose
operations have one input but many outputs.

\begin{construct}[Anti-plus construction]
  We let $\overline{\bbDelta}_{\cat{O}}^{\mathrm{pre}}$ denote the
  Grothendieck construction of
  $\restr{\yo_{\infcats}(\cat{O}^{\mathrm{act}})}_{\bbDelta}\circ\op{(-)}$,
  and define a locally full sub-$(\infty,1)$-category
  $\overline{\bbDelta}_{\cat{O}}$ by imposing the same conditions as
  in the plus construction.
\end{construct}

\begin{scholium}
  There is an equivalence of $(\infty,1)$-categories between Segal
  ${\op{\overline{\bbDelta}_{\cat{O}}}}^{\natural}$-objects and
  $(\infty,1)$-functors $\func{p}\colon\cat{P}\to\op{\cat{O}}$ such
  that $\op{\func{p}}$ is a weak Segal fibration.
\end{scholium}

\begin{definition}[Reduced pattern]
  A combinatorial pattern $\cat{O}$ is \textbf{reduced} if
  $\cat{O}_{0}^{\mathrm{el}}$ is the terminal $\infty$-groupoid. We
  denote its (essentially) unique object $\zeroelem$.
\end{definition}

\begin{definition}[Unitality]
  Let $\cat{O}$ be a reduced pattern. A
  ${\op{\overline{\bbDelta}_{\cat{O}}}}^{\natural}$-object $\func{X}$
  is \textbf{anti-unital} if it is local with respect to the
  (op-)morphism $\corol{\zeroelem}\leftarrow\edge$.
\end{definition}

\begin{remark}
  The $(\infty,1)$-category of anti-unital Segal
  ${\op{\overline{\bbDelta}_{\cat{O}}}}^{\natural}$-objects in an
  $(\infty,1)$-category $\cat{C}$ is evidently a localisation of
  $\seg{{\op{\overline{\bbDelta}_{\cat{O}}}}^{\natural}}(\cat{C})$. It
  can also be seen a category of Segal objects for a different
  pattern.

  The pattern for anti-unital
  ${\op{\overline{\bbDelta}_{\cat{O}}}}^{\natural}$-objects is the
  same as ${\op{\overline{\bbDelta}_{\cat{O}}}}^{\natural}$, with
  $\corol{\zeroelem}$ removed from the elementaries.
\end{remark}

\begin{lemma}
  The restriction morphism
  $\seg{{\op{\overline{\bbDelta}_{\cat{O}}}}^{\natural,\mathrm{unit.}}}(\cat{C})
  \to\seg{{\op{\overline{\bbDelta}_{\cat{O}}}}^{\natural}_{\edge/}}(\cat{C})$
  (taking underlying $\infty$-categories) preserves Segal
  equivalences, and thus admits a right-adjoint (by lax extension).
\end{lemma}

We now give an alternate construction of the cartesian structure as a
${\op{\bbDelta_{\cat{O}}}}^{\natural}$-object, in the spirit
of~\cite{lurie17:_higher_algeb}.

\begin{construct}
  Let $\shname{U}_{\leq-1}$ be the subobject classifier
  ($(-1)$-truncated universe) in the $(\infty,1)$-topos
  $\prshvs{\bbDelta_{\cat{O}}}$. Define a presheaf
  $\widetilde{\func{Env}}(\cat{C}^{\times})$ over
  $\func{Env}(\mathbb{1})$ by the specification that
  $\hom(\shname{K},\widetilde{\func{Env}}(\cat{C}^{\times})) \simeq
  \hom(\cat{K}\times_{\cat{Env}(\mathbb{1})}\cat{Env}(\shname{U}_{\leq-1}),
  \cat{C})$ for any presheaf $\func{K}$. Note that the canonical
  (mono)morphism
  $\mathtt{true}\colon\mathbb{1}\rightarrowtail\shname{U}_{\leq-1}$
  induces a transformation
  $\widetilde{\func{Env}}(\cat{C}^{\times})\to\cat{C}$.

  When $\shname{K}$ is $\yo\edge$, and the transformation
  $\edge\to\func{Env}(1)$ selects an object $O$ of
  $\cat{O}^{\mathrm{act}}$, we obtain functors from the poset of
  subobjects of $O$ to $\cat{C}$. We define
  $\func{Env}(\cat{C}^{\times})$ to be the subfunctor of
  $\widetilde{\func{Env}}(\cat{C}^{\times})$ on the functors as above
  which exhibit their image as product of the elementaries in their
  source.
\end{construct}

\begin{pros}
  The structure morphism
  $\cat{Env}(\cat{C}^{\times})\to\cat{Env}(\mathbb{1})$ is the image
  of a cocartesian fibration of
  ${\op{\bbDelta_{\cat{O}}}}^{\natural}$-objects.
\end{pros}

\begin{proof}
  As in~\cite[Proposition 2.4.1.5.]{lurie17:_higher_algeb}.
\end{proof}

\subsection{Straightening cocartesian fibrations}
\label{sec:unstr-cocart-fibr}

\begin{definition}[Weak cartesian structure]
  Let $\func{X}$ be a Segal
  ${\op{\bbDelta_{\cat{O}}}}^{\natural}$-$\infty$-groupoid, and let
  $\cat{D}$ be a Segal
  ${\op{\bbDelta_{\cat{O}}}}_{\edge/}^{\natural}$-$\infty$-groupoid. A
  \textbf{lax cartesian structure} from $\func{O}$ to $\cat{D}$ is a
  morphism $\cat{Env}(\func{X})\to\cat{D}$ that takes decompositions
  to products.

  It is a \textbf{strong} cartesian structure if it induces an
  equivalence of underlying $(\infty,1)$-categories.
\end{definition}

\begin{lemma}
  The transformation $\cat{Env}(\cat{C}^{\times})\to\cat{C}$ is a
  strong cartesian structure.
\end{lemma}

\begin{proof}
  As in~\cite[Proposition 2.4.1.5.]{lurie17:_higher_algeb}.
\end{proof}

\begin{pros}
  There is an equivalence of $(\infty,1)$-categories between lax
  cartesian structures from $\func{X}$ to $\cat{D}$ and morphisms
  $\func{X}\to\cat{D}^{\times}$.
\end{pros}

\begin{proof}
  As in~\cite[Proposition 2.4.1.7]{lurie17:_higher_algeb}.
\end{proof}

\begin{remark}
  Consider a lax cartesian structure
  $\varphi\colon\cat{Env}(\func{X})\to\cat{D}$ in the case where
  $\cat{D}$ is $\infcats$ (or $\infgrpds$). It is classified by a
  cocartesian fibration
  $\func{p}\colon\Phi\coloneqq\grothco\varphi\to\cat{Env}(\func{X})$. We
  will say that a cocartesian fibration over $\cat{Env}(\func{X})$ is
  \textbf{lax cartesian} if the $\infty$-functor
  $\cat{Env}(\func{X})\to\infcats$ that it classifies is a lax
  cartesian structure.
\end{remark}

\begin{lemma}
  A cocartesian fibration over $\func{Env}(\func{X})$ is lax cartesian
  monoidal if and only if it is in the essential image of the functor
  $\func{Env}$.
\end{lemma}

\begin{proof}
  By~\cite{haugseng21} or~\cite{barkan22:_envel_algeb_patter}, a
  cocartesian fibration over $\func{Env}(\func{X})$ is in the image of
  $\func{Env}$ if and only if it is equifibred (or cartesian) as a
  natural transformation.
\end{proof}

We thus obtain the following (un)straightening correspondence.

\begin{corlr}
  For every Segal ${\op{\bbDelta_{\cat{O}}}}^{\natural}$-$\infty$-groupoid
  $\func{X}$, there is an equivalence of $(\infty,1)$ categories
  between cocartesian fibrations over $\func{X}$ and morphisms
  $\func{X}\to\infcats^{\times}$.\qed{}
\end{corlr}

\begin{remark}
  The straightening/unstraightening equivalence for Segal
  ${\op{\bbDelta_{\cat{O}}}}^{\natural}$-$\infty$-groupoids takes as
  input the one for
  ${\op{\bbDelta_{\cat{O}}}}_{\eta/}^{\natural}$-$\infty$-groupoids. This
  is reminiscent of the straightening correspondence for Segal
  ${\op{(\bbDelta^{n})}}^{\natural}$-$\infty$-groupoids
  of~\cite{nuiten21:_segal} which is inductive on $n$.
\end{remark}

\printbibliography

~

{\footnotesize 
\textsc{David Kern, IMAG,
  Université de Montpellier, CNRS, Montpellier, France}

\textit{Email address}:
\href{mailto:david.kern@umontpellier.fr}{\texttt{david.kern@umontpellier.fr}}

\emph{URL}:
\href{https://dskern.github.io/}{\texttt{https://dskern.github.io/}}}

\end{document}